\pdfoutput=1

\documentclass{svjour3}                     

\smartqed  

\makeatletter
\renewcommand*{\@thmcounterend}{}
\makeatother
\spnewtheorem{assumption}{Assumption}{\bfseries}{\itshape}

\usepackage{cite}
\usepackage[numbers,sort&compress]{natbib}

\usepackage{stmaryrd} 
\usepackage{amsmath, amsfonts, amssymb}

\usepackage{hyperref}
\hypersetup{colorlinks,
citecolor=blue
}

\usepackage[titletoc,toc,title]{appendix}
\usepackage{tikz}
\usepackage{graphicx, color, bm}
\usepackage{subfig}
\usepackage{pgfplots}
\usepackage{pgfplotstable}
\definecolor{markercolor}{RGB}{124.9, 255, 160.65}

\renewcommand{\hat}{\widehat}
\renewcommand{\tilde}{\widetilde}

\newcommand{\diag}[1]{{\rm diag}\LRp{#1}}
\newcommand{\td}[2]{\frac{{\rm d}#1}{{\rm d}{\rm #2}}}
\newcommand{\pd}[2]{\frac{\partial#1}{\partial#2}}
\newcommand{\nor}[1]{\left\| #1 \right\|}
\newcommand{\LRp}[1]{\left( #1 \right)}

\newcommand{\LRc}[1]{\left\{ #1 \right\}}

\newcommand{\LRl}[1]{\left. #1 \right|}

\newcommand{\jump}[1] {\ensuremath{\llbracket#1\rrbracket}}
\newcommand{\avg}[1] {\ensuremath{\LRc{\!\LRc{#1}\!}}}

\newcommand{\fnt}[1]{\bm{\mathsf{ #1}}}
\graphicspath{{./figs/}}

\date{}

\author{Jesse Chan \and Mario J.\ Bencomo \and David C.\ Del Rey Fern\'{a}ndez}
\institute{J.\ Chan and M.\ Bencomo \at
	Rice University\\
	Department of Computational and Applied Mathematics\\
	Houston, TX, United States\\
              \email{jesse.chan@rice.edu}\\
              \email{mario.j.bencomo@rice.edu}           
           \and
           D.\ C.\ Del Rey Fern\'{a}ndez \at
              National Institute of Aerospace and Computational Aerosciences Branch\\
              NASA Langley Research Center\\
              Hampton, VA, United States\\
\email{dcdelrey@gmail.com}
}

\title{Mortar-based entropy-stable discontinuous Galerkin methods on non-conforming quadrilateral and hexahedral meshes}
\titlerunning{Mortar-based entropy-stable DG methods}

\graphicspath{{./figs/}}

\begin{document}

\maketitle

\begin{abstract}
High-order entropy-stable discontinuous Galerkin (DG) methods for nonlinear conservation laws reproduce a discrete entropy inequality by combining entropy conservative finite volume fluxes with summation-by-parts (SBP) discretization matrices.  In the DG context, on tensor product (quadrilateral and hexahedral) elements, SBP matrices are typically constructed by collocating at Lobatto quadrature points.  Recent work has extended the construction of entropy-stable DG schemes to collocation at more accurate Gauss quadrature points \cite{chan2018efficient}.  

In this work, we extend entropy-stable Gauss collocation schemes to non-conforming meshes.  Entropy-stable DG schemes require computing entropy conservative numerical fluxes between volume and surface quadrature nodes.  On conforming tensor product meshes where volume and surface nodes are aligned, flux evaluations are required only between ``lines'' of nodes.  However, on non-conforming meshes, volume and surface nodes are no longer aligned, resulting in a larger number of flux evaluations.  We reduce this expense by introducing an entropy-stable mortar-based treatment of non-conforming interfaces via a face-local correction term, and provide necessary conditions for high-order accuracy.  Numerical experiments in both two and three dimensions confirm the stability and accuracy of this approach.  
\end{abstract}

\section{Introduction}
Discretely entropy stability has emerged as a methodology for designing high-order schemes for nonlinear conservation laws.  Entropy-stable discretizations ensure the satisfaction of a semi-discrete entropy inequality by combining specific finite volume numerical fluxes with summation-by-parts (SBP) discretization matrices.\footnote{Recent work on space-time discretizations \cite{friedrich2018entropy} and relaxation time-stepping methods \cite{ranocha2020relaxation} has extended the semi-discrete entropy inequality to the fully discrete setting.} Compared to traditional high-order methods, the resulting schemes demonstrate significantly improved robustness in the presence of under-resolved solution features such as shocks or turbulence while retaining high-order accuracy.  

Entropy-stable discontinuous Galerkin (DG) methods were originally constructed for quadrilateral and hexahedral meshes based on nodal collocation at Lobatto quadrature points \cite{fisher2013high, carpenter2014entropy, gassner2016split}.  Entropy-stable schemes were later extended to simplicial and more general elements using tailored volume and surface quadrature rules \cite{chen2017entropy, crean2018entropy}.   More general quadrature rules were addressed in \cite{chan2017discretely, chan2018discretely, chan2018efficient, chan2019skew}, including entropy-stable collocation schemes on quadrilateral and hexahedral elements based on more accurate Gauss quadrature rules and generalized SBP operators \cite{chan2018efficient}.  

The work presented here focuses on geometrically non-conforming meshes.  Such meshes may arise when applying domain decomposition techniques to a complex geometry (e.g. meshing sub-domains independently) \cite{bernardi1993domain} or performing local mesh refinement. Energy stable treatments of non-conforming interfaces have been widely explored in the context of SBP discretizations \cite{mattsson2010stable, nissen2015stable, wang2016high, wang2018improved, almquist2019order, aalund2019encapsulated}, and have recently been extended to entropy stable schemes. Entropy-stable Lobatto collocation schemes have been constructed on non-conforming meshes in \cite{friedrich2017entropy} using SBP projection operators.  In this work, we extend entropy-stable Gauss collocation to non-conforming quadrilateral and hexahedral meshes.  While it is straightforward to construct Gauss collocation schemes on non-conforming meshes, the treatment of non-conforming interfaces results in signficantly increased computational costs.  We reduce such costs by adopting a mortar-based treatment of non-conforming interfaces.  Moreover, while the cost of the proposed scheme is similar to that of  \cite{friedrich2017entropy} in 2D, the mortar-based approach is more computationally efficient for both Lobatto and Gauss collocation schemes on 3D non-conforming hexahedral meshes.  

The paper is organized as follows: Section~\ref{sec:0} briefly reviews the derivation of an entropy inequality for a system of nonlinear conservation laws.  Section~\ref{sec:1} introduces ``hybridized'' SBP operators as a unified way to treat both Lobatto and Gauss collocation schemes on conforming meshes of tensor product elements.  Section~\ref{sec:2} describes a naive extension to non-conforming meshes and illustrates why this formulation results in an increase in computational costs.  Section~\ref{sec:3} introduces a mortar-based formulation on curved meshes which addresses such costs, and characterizes the accuracy of the resulting formulations.  We conclude with numerical validation of theoretical results in Section~\ref{sec:4}.

\section{Entropy stability for systems of nonlinear conservation laws}
\label{sec:0}
We are interested in the numerical approximation of solutions to systems of nonlinear conservation laws
\begin{equation}
\pd{\bm{u}}{t} + \sum_{i=1}^d \pd{\bm{f}_i\LRp{\bm{u}}}{x_i} = 0.
\label{eq:nonlinpde}
\end{equation}
Here, $\bm{u}$ denotes the conservative variables and $\bm{f}_i(\bm{u})$ are nonlinear fluxes.  We briefly review entropy stability for systems of conservation laws in $d$ dimensions.  We assume there exists a convex scalar entropy $S(\bm{u})$ associated with (\ref{eq:nonlinpde}).  We then define the entropy variables $\bm{v}(\bm{u})$ as the gradient of the entropy $S(\bm{u})$ with respect to the conservative variables 
\[
\bm{v} = \pd{S(\bm{u})}{\bm{u}}.  
\]
For $S(\bm{u})$ convex, $\bm{v}(\bm{u})$ defines an invertible mapping between the conservative and entropy variables, whose inverse (from entropy to conservative variables) we denote by $\bm{u}(\bm{v})$.  Viscosity solutions to (\ref{eq:nonlinpde}) satisfy an integrated form of the entropy inequality \cite{dafermos2005compensated}
\begin{equation}
\int_{\Omega} \pd{S(\bm{u})}{t} + \int_{\partial \Omega} \sum_{i=1}^d n_i \LRp{\bm{v}^T\bm{f}_i(\bm{u}) - \psi_i(\bm{u})} \leq 0,
\label{eq:weakentropyineq}
\end{equation}
where $\psi_i(\bm{u})$ denotes the $i$th entropy potential, $\partial \Omega$ denotes the boundary of $\Omega$ and ${n}_i$ denotes the $i$th component of the outward normal on $\partial \Omega$.  This can be interpreted as implying that the time rate of change of entropy is bounded by the entropy flux through the boundary.

\section{Entropy-stable collocation DG methods and hybridized SBP operators}
\label{sec:1}

In this section, we summarize the work of \cite{chan2018efficient} on the construction of entropy-stable collocation DG methods based on generalized summation by parts (GSBP) operators.  These constructions are applicable to collocation schemes based on either Lobatto and Gauss nodes.  

\subsection{On notation}

The notation in this paper is motivated by notation in \cite{crean2018entropy, fernandez2019entropy}.
Unless otherwise specified, vector and matrix quantities are denoted using lower and upper case bold font, respectively.  We also denote spatially discrete quantities using a bold sans serif font. Finally, the output of continuous functions evaluated over discrete vectors is interpreted as a discrete vector. For example, if $\fnt{x}$ denotes a vector of point locations, i.e., $(\fnt{x})_i = \bm{x}_i$, then $u(\fnt{x})$ is interpreted as the vector 
\[	({u}(\fnt{x}))_i = {u}(\bm{x}_i).
\]
Similarly, if $\fnt{u} = {u}(\fnt{x})$, then ${f}(\fnt{u})$ corresponds to the vector
\[
	({f}(\fnt{u}))_i = {f}(u(\bm{x}_i)).
\]
Vector-valued functions are treated similarly. For example, given a vector-valued function $\bm{f}:\mathbb{R}^n\rightarrow \mathbb{R}^n$ and vector of points $\fnt{x}$, $\LRp{\bm{f}(\fnt{x})}_i = \bm{f}(\bm{x}_i)$.

\subsection{Hybridized SBP operators in 1D}

We begin by introducing collocation discretization matrices on the reference interval $[-1,1]$.  We assume the solution is collocated at $(N+1)$ quadrature points $x_i$ with associated quadrature weights $w_i$, and consider primarily Lobatto or Gauss quadrature points.  The collocation assumption is equivalent to approximating the solution using a degree $N$ Lagrange basis $\ell_j(x)$ at the $(N+1)$ quadrature points.  

We define mass and integrated differentiation matrices $\hat{\fnt{M}}_{\rm 1D},\hat{\fnt{Q}}_{\rm 1D}$ on the reference interval
\[
\LRp{\hat{\fnt{M}}_{\rm 1D}}_{ij} = \int_{-1}^1 \ell_i(x)\ell_j(x), \qquad \LRp{\hat{\fnt{Q}}_{\rm 1D}}_{ij} = \int_{-1}^1 \pd{\ell_j}{x}\ell_i(x).
\]
We assume that all integrals are computed using the collocated quadrature rule, such that 
\begin{gather*}
\LRp{\hat{\fnt{M}}_{\rm 1D}}_{ij} = \int_{-1}^1 \ell_i(x)\ell_j(x) \approx \sum_{k=1}^{N+1} \ell_i(x_k)\ell_j(x_k) w_k = \delta_{ij} w_i\\
\LRp{\hat{\fnt{Q}}_{\rm 1D}}_{ij} = \int_{-1}^1 \pd{\ell_j}{x}\ell_i(x) \approx \sum_{k=1}^{N+1} \ell_i(x_k)\LRl{\pd{\ell_j}{x}}_{x_k} w_k.  
\end{gather*}
In other words, the mass matrix is diagonal with entries equal to the quadrature weights.  Since the integrands of $\hat{\fnt{M}}_{\rm 1D}$ are degree $2N$ polynomials, the collocation approximation of $\hat{\fnt{M}}_{\rm 1D}$ is exact for Gauss quadrature, but not for Lobatto quadrature.  For both Gauss and Lobatto quadrature, the matrix ${\hat{\fnt{Q}}_{\rm 1D}}$ is exact under collocation quadrature.

We introduce the $2\times (N+1)$ matrix $\fnt{E}_{\rm 1D}$ which interpolates values at collocation nodes to values at the endpoints $x = -1$ and $x=1$.  This matrix is defined entrywise as
\[
\LRp{\fnt{E}_{\rm 1D}}_{1i} = \ell_i(-1), \qquad  \LRp{\fnt{E}_{\rm 1D}}_{2i} = \ell_i(1).
\]
The mass and differentiation matrices satisfy a generalized summation by parts property \cite{fernandez2014generalized}
\begin{equation}
\hat{\fnt{Q}}_{\rm 1D} + \hat{\fnt{Q}}_{\rm 1D}^T = \fnt{E}_{\rm 1D}^T \hat{\fnt{B}}_{\rm 1D} \fnt{E}_{\rm 1D}, \qquad \hat{\fnt{B}}_{\rm 1D} = \begin{bmatrix}-1 & \\ & 1\end{bmatrix}.
\label{eq:gsbp}
\end{equation}
The GSBP property holds for both Lobatto and Gauss nodes, and switching between these two nodal sets simply requires redefining the matrices $\hat{\fnt{Q}}_{\rm 1D}, \fnt{E}_{\rm 1D}$.  For Gauss nodes, $\fnt{E}_{\rm 1D}$ is dense.  For Lobatto nodes, since the collocation nodes include boundary points, the interpolation matrix $\fnt{E}_{\rm 1D}$ reduces to the matrix which extracts nodal values associated with the left and right endpoints
\[
\fnt{E}_{\rm 1D} = \begin{bmatrix}
1 & 0 & \ldots & 0\\
0 & \ldots & 0 & 1
\end{bmatrix}.
\]

It is possible to construct energy stable high-order discretizations of linear hyperbolic systems using GSBP operators based on Gauss nodes \cite{fernandez2014generalized}.  However, for nonlinear conservation laws, the presence of the dense $\fnt{E}$ matrix in the GSBP property (\ref{eq:gsbp}) complicates the imposition of boundary conditions and computation of inter-element numerical fluxes \cite{crean2018entropy, chan2017discretely, chan2018efficient}.  This can be avoided by using ``hybridized'' (also referred to as decoupled) SBP operators \cite{chan2017discretely, chenreview, chan2019skew}, which are defined as the block matrix 
\[
\hat{\fnt{Q}}_{h,{\rm 1D}} = \frac{1}{2}\begin{bmatrix}
\hat{\fnt{Q}}_{\rm 1D}-\hat{\fnt{Q}}_{\rm 1D}^T & \fnt{E}_{\rm 1D}^T\hat{\fnt{B}}_{\rm 1D}\\
-\hat{\fnt{B}}_{\rm 1D}\fnt{E}_{\rm 1D} & \hat{\fnt{B}}_{\rm 1D}
\end{bmatrix}.
\]
The hybridized SBP operator satisfies a block form of the SBP property
\begin{equation}
\hat{\fnt{Q}}_{h,{\rm 1D}}  + \hat{\fnt{Q}}_{h,{\rm 1D}} ^T = \hat{\fnt{B}}_{h,{\rm 1D}}, \qquad \hat{\fnt{B}}_{h,{\rm 1D}} =  \begin{bmatrix}
\fnt{0} & \\
&\hat{\fnt{B}}_{\rm 1D}\end{bmatrix}.
\label{eq:hsbp}
\end{equation}
where $\fnt{E}_{\rm 1D}$ does not appear in the block boundary matrix on the right hand side. Here, we have used $\fnt{0}, \fnt{1}$ to denote a matrix or vector of zeros or ones, where the size is inferred from context. Note that the matrix $\hat{\fnt{Q}}_{h,{\rm 1D}}$ also satisfies
\begin{align}
\hat{\fnt{Q}}_{h,{\rm 1D}}\fnt{1} &= 
\frac{1}{2}\begin{bmatrix}
\hat{\fnt{Q}}_{\rm 1D}\fnt{1}-\hat{\fnt{Q}}_{\rm 1D}^T\fnt{1} + \fnt{E}_{\rm 1D}^T\hat{\fnt{B}}_{\rm 1D}\fnt{1}\\
-\hat{\fnt{B}}_{\rm 1D}\fnt{E}_{\rm 1D}\fnt{1} - \hat{\fnt{B}}_{\rm 1D}\fnt{1}
\end{bmatrix}= 
\frac{1}{2}\begin{bmatrix}
-\hat{\fnt{Q}}_{\rm 1D}^T\fnt{1} + \fnt{E}_{\rm 1D}^T\hat{\fnt{B}}_{\rm 1D}\fnt{1}\\
\fnt{0}
\end{bmatrix} \nonumber\\
&=
\frac{1}{2}\begin{bmatrix}
\hat{\fnt{Q}}_{\rm 1D}\fnt{1}\\
\fnt{0}
\end{bmatrix} = \fnt{0}
\label{eq:Qh1}
\end{align}
where we have used that $\fnt{E}_{\rm 1D}\fnt{1} = \fnt{1}$ (since $\fnt{E}_{\rm 1D}$ is a high-order accurate boundary interpolation matrix), $\hat{\fnt{Q}}_{\rm 1D}\fnt{1} = \fnt{0}$ (since $\hat{\fnt{Q}}_{\rm 1D}$ is a differentiation matrix), and the GSBP property (\ref{eq:gsbp}).

\subsection{Hybridized SBP operators in higher dimensions}

The formulation (\ref{eq:weakentropyineq}) can be extended to higher dimensions through a tensor product construction.  We assume that both volume and surface quadrature nodes are constructed from one-dimensional $(N+1)$-point Lobatto or Gauss quadrature rules.

For simplicity, we illustrate this for 2D quadrilateral elements (the extension to 3D hexahedral elements is straightforward).  Let $\hat{\fnt{M}}_{\rm 1D}, \hat{\fnt{Q}}_{\rm 1D}$ denote one-dimensional generalized SBP norm (mass) and differentiation matrices on the reference interval, and let $\fnt{E}_{\rm 1D}$ denote the 1D face interpolation matrix.  We define multi-dimensional  mass and differentiation matrices in terms of Kronecker products  
\begin{equation}
\hat{\fnt{Q}}_1 = \hat{\fnt{Q}}_{\rm 1D} \otimes \hat{\fnt{M}}_{\rm 1D}, \qquad \hat{\fnt{Q}}_2  = \hat{\fnt{M}}_{\rm 1D} \otimes \hat{\fnt{Q}}_{\rm 1D}, \qquad \hat{\fnt{M}} = \hat{\fnt{M}}_{\rm 1D} \otimes \hat{\fnt{M}}_{\rm 1D}.
\label{eq:kron}
\end{equation}
We also construct 2D face interpolation matrices from Kronecker products.  Let $\fnt{E}_{\rm 1D}$ denote the one-dimensional face interpolation matrix, and let $\hat{\fnt{B}}_{\rm 1D}$ denote the 1D boundary matrix in (\ref{eq:gsbp}).  For a specific ordering of the face points, the two-dimensional face interpolation matrix $\fnt{E}$ and reference boundary matrices $\hat{\fnt{B}}_1, \hat{\fnt{B}}_2$ are given by
\begin{equation}
\fnt{E} = \begin{bmatrix}
\fnt{E}_{\rm 1D} \otimes \fnt{I}_{N+1}\\
\fnt{I}_{N+1} \otimes \fnt{E}_{\rm 1D} 
\end{bmatrix}, \qquad 
\hat{\fnt{B}}_1 = \begin{bmatrix}
\hat{\fnt{B}}_{\rm 1D} \otimes \hat{\fnt{M}}_{\rm 1D} & \\
& \fnt{0}
\end{bmatrix}, \qquad 
\hat{\fnt{B}}_2 = \begin{bmatrix}
\fnt{0} &\\
& \hat{\fnt{M}}_{\rm 1D} \otimes \hat{\fnt{B}}_{\rm 1D} 
\end{bmatrix},
\label{eq:EBdefs}
\end{equation}
where $\fnt{I}_{N+1}$ is the $(N+1)\times (N+1)$ identity matrix.  Recall that $\hat{\fnt{M}}_{\rm 1D}$ is the diagonal matrix of quadrature weights, so $\hat{\fnt{B}}_i$ are diagonal.  

The construction of operators in (\ref{eq:kron}) and (\ref{eq:EBdefs}) correspond to the use of tensor product points quadrature points and \emph{aligned} surface points.  For example, if Gauss points are used for volume quadrature, we assume that Gauss points are also used for surface quadrature.  This ensures that the operators $\hat{\fnt{Q}}_i, \hat{\fnt{B}}_i, \fnt{E}$ satisfy the analogous higher dimensional generalized SBP property $\hat{\fnt{Q}}_i = \fnt{E}^T\hat{\fnt{B}}_i\fnt{E} - \hat{\fnt{Q}}_i^T$.

The 2D differentiation and interpolation matrices $\hat{\fnt{Q}}_i, \fnt{E}$ can now be used to construct 2D hybridized SBP operators.  Let $\hat{\fnt{Q}}_{i,h}$ denote the hybridized SBP operator for the $i$th coordinate on the reference element, where $\hat{\fnt{Q}}_{i,h}$ is defined as
\begin{align}
\hat{\fnt{Q}}_{i,h} = \frac{1}{2}\begin{bmatrix}
\hat{\fnt{Q}}_i - \hat{\fnt{Q}}_i^T & \fnt{E}^T\hat{\fnt{B}}_i\\
-\hat{\fnt{B}}_i\fnt{E} & \hat{\fnt{B}}_i
\end{bmatrix}.
\label{eq:hsbpdef}
\end{align}
These matrices satisfy the following multi-dimensional properties:
\begin{lemma}
\label{lemma:hsbprefprops}
Let $\hat{\fnt{Q}}_{i}$ be defined as in (\ref{eq:kron}) and let $\hat{\fnt{Q}}_{i,h}$ be defined as in (\ref{eq:hsbpdef}).  Then, 
\begin{align}
\hat{\fnt{Q}}_{i,h} + \hat{\fnt{Q}}_{i,h}^T = \begin{bmatrix}
\fnt{0} & \\
& \hat{\fnt{B}}_i
\end{bmatrix}, \qquad \hat{\fnt{Q}}_{i,h}\fnt{1} = \fnt{0}.
\end{align}
\end{lemma}
The proof is found in \cite{chan2018efficient}.  The SBP property follows from the construction of $\hat{\fnt{Q}}_{i,h}$, and $\hat{\fnt{Q}}_{i,h}\fnt{1} = \fnt{0}$ is a consequence of properties of the Kronecker product and the same arguments used to derive (\ref{eq:Qh1}).  The construction of 3D differentiation and interpolation matrices proceeds similarly, and the matrices satisfy the same properties as described in Lemma~\ref{lemma:hsbprefprops}.  


We now construct an entropy conservative formulation on an unstructured curved mesh.  
Suppose the domain $\Omega$ is decomposed into non-overlapping elements $D^k$ which are images of the reference mapping, such that $D^k$ is a differentiable mapping of the reference element $\hat{D} = [-1,1]^d$.  Let ${x}_i$ denote the $i$th physical coordinate on $D^k$ and let $\hat{{x}}_j$ denote the $j$th reference coordinate.  The mapping between reference and physical element induces scaled geometric terms 
\[
g_{ij} = J\pd{\hat{x}_j}{{x}_i},
\]
where $J$ denotes the determinant of the Jacobian of the physical-to-reference mapping.  We refer to $J$ as the Jacobian from here onwards.
These geometric terms also relate the normal vectors on reference and physical elements.  

Let $\hat{\bm{n}}\hat{J}_f$ denote the scaled outward normal vector on the reference element $\hat{D}$, where $\hat{J}_f$ denotes the Jacobian for the corresponding reference face.  For quadrilateral and hexahedral elements, $\hat{J}_f = 1$.  Then, the (scaled) physical normal vectors on $D^k$ are related to $\hat{\bm{n}}\hat{J}_f$ through
\begin{equation}
{n}_iJ_f = \sum_{j=1}^d{g}_{ij}\hat{{n}}_j\hat{J}_f.
\label{eq:nJ_Gnhat}
\end{equation}
The relation (\ref{eq:nJ_Gnhat}) is known as Nanson's formula in continuum mechanics. 

We assume that the mesh is watertight (well-constructed), such that the scaled normal vectors are equal and opposite across each shared face between two elements.  We also assume that the geometric terms ${g}_{ij}$ are constructed such that they are polynomials of degree less than or equal to $N$ and satisfy a discrete geometric conservation law (GCL)
\begin{equation}
\sum_{j=1}^d \pd{}{\hat{x}_j} g_{ij} = 0.
\label{eq:dgcl}
\end{equation}
The discrete GCL is satisfied automatically for isoparametric mappings in 2D, and there exist several techniques to enforce the satisfaction of a discrete GCL on various three-dimensional domains \cite{thomas1979geometric, kopriva2006metric, crean2018entropy, chan2018discretely, kozdon2018energy, kopriva2019free}.  We will discuss specific approaches for non-conforming meshes in Section~\ref{sec:mapped} and Appendix~\ref{app:A}.

We can now assemble physical mass and differentiation matrices on $D^k$ from reference mass and differentiation matrices using the chain rule.  Let $\hat{\fnt{Q}}_{j,h}$ denote the $j$th reference differentiation matrix on $\hat{D}$.  We can define the physical differentiation matrices $\fnt{Q}_{i,h}$ on $D^k$ via the skew-symmetric splitting
\begin{equation}
\fnt{Q}_{i,h} = \frac{1}{2}\sum_{j=1}^d \diag{\fnt{g}_{ij}}\hat{\fnt{Q}}_{j,h} + \hat{\fnt{Q}}_{j,h} \diag{\fnt{g}_{ij}},
\label{eq:curvedQ}
\end{equation}
where $\fnt{g}_{ij}$ denotes the vector of values of $g_{ij} = J\pd{\hat{x}_j}{x_i}$ at both volume and surface points.  Since $g_{ij}$ is assumed to be a polynomial of degree less than or equal to $N$, $\fnt{g}_{ij}$ is constructed using polynomial interpolation.

Because the reference matrices $\hat{\fnt{Q}}_{j,h}$ satisfy the reference SBP properties (\ref{eq:hsbp}), one can show \cite{chan2018discretely} that the physical matrices $\fnt{Q}_{i,h}$ satisfy analogous SBP properties.  Let $\circ$ denote the Hadamard product of matrices or vectors. Then, we have the following lemma:
\begin{lemma}
\label{lemma:Qhprops}
Suppose the geometric terms $\fnt{g}_{ij}$ satisfy the discrete GCL (\ref{eq:dgcl}), the normals are constructed via (\ref{eq:nJ_Gnhat}), and that $\fnt{Q}_{i,h}$ is constructed using (\ref{eq:curvedQ}). Define $\fnt{B}_i = \diag{\fnt{n}_i \circ \fnt{w}_f}$, where the entries of $\fnt{n}_i$ are values of the $i$th component of the scaled normals $n_iJ_f$ at face points, and $\fnt{w}_f$ contains face quadrature weights. Then, 
\[
\fnt{Q}_{i,h} + \fnt{Q}_{i,h}^T = \begin{bmatrix}
\fnt{0} &\\
& \fnt{B}_i \end{bmatrix}, \qquad \fnt{Q}_{i,h}\fnt{1} = \fnt{0},
\]
\end{lemma}
The proof is given in \cite{chan2018discretely}.  

\subsection{Entropy conservative formulations on conforming meshes of mapped elements}
\label{sec:mapped}

The aforementioned matrices can now be used to construct high-order accurate entropy-stable and entropy-conservative schemes.  We first introduce an entropy conservative numerical flux \cite{tadmor1987numerical}.  Let $\bm{u}_L, \bm{u}_R$ denote left and right states in the conservative variables.  Then, an entropy conservative numerical flux is a vector-valued function $\bm{f}_{i,S}$ for $i = 1,\ldots,d$ which satisfies the following three properties
\begin{gather*}
\bm{f}_{i,S}\LRp{\bm{u},\bm{u}} = \bm{f}_i(\bm{u}), \qquad \text{(consistency)}\\
\bm{f}_{i,S}\LRp{\bm{u}_L,\bm{u}_R} = \bm{f}_{i,S}(\bm{u}_R,\bm{u}_L), \qquad \text{(symmetry)}\\
\LRp{\bm{v}_L-\bm{v}_R}^T\bm{f}_{i,S}\LRp{\bm{u}_L,\bm{u}_R} = \psi_i(\bm{u}_L) - \psi_i(\bm{u}_R), \qquad \text{(conservation)}.
\end{gather*}
Here, $\bm{v}_L, \bm{v}_R$ are the entropy variables evaluated at $\bm{u}_L, \bm{u}_R$, and $\psi_i$ is the $i$th entropy potential which appears in (\ref{eq:weakentropyineq}).  
Then, an entropy conservative scheme on a mapped element is given by
\begin{gather}
{\fnt{M}}\td{\fnt{u}_h}{t} + \begin{bmatrix} \fnt{I} \\ \fnt{E} \end{bmatrix}^T
\sum_{i=1}^d \LRp{2{\fnt{Q}}_{i,h} \circ \fnt{F}_i}\fnt{1} + \fnt{E}^T{\fnt{B}}_i\LRp{\fnt{f}_i^*-\bm{f}_i(\tilde{\fnt{u}}_f)} = 0 \label{eq:esdg}\\
\LRp{\fnt{F}_i}_{jk} = \bm{f}_{i,S}\LRp{\tilde{\fnt{u}}_j,\tilde{\fnt{u}}_k}, \qquad \fnt{f}_i^* = \bm{f}_{i,S}\LRp{\tilde{\fnt{u}}_f,\tilde{\fnt{u}}_f^+}, \nonumber
\end{gather}
where $\fnt{u}_h(t)$ denotes the coefficients of the discrete solution on $D^k$ and $\fnt{M} = \hat{\fnt{M}} \circ \diag{\fnt{J}}$ is the diagonal mass matrix scaled by the determinant of the geometric Jacobian $J$ evaluated at volume quadrature points. 
The discrete vectors of ``entropy-projected'' conservative variables with tildes (e.g., $\tilde{\fnt{u}}$) are constructed by first evaluating the interpolated entropy variables
\[
	\fnt{v}_f = \fnt{E}\bm{v}(\fnt{u}_h)
\]
then re-evaluating the conservative variables in terms of $\fnt{v}_f$
\[
	\tilde{\fnt{u}}_f = \bm{u}(\fnt{v}_f), \quad  
	\tilde{\fnt{u}} = \begin{bmatrix} \fnt{u}\\ \tilde{\fnt{u}}_f \end{bmatrix}
\]
The flux matrix $\fnt{F}_i$ contains evaluation of the entropy conservative numerical flux at different pairs of solution values at nodal (both volume and surface quadrature) points. Finally, the exterior state $\tilde{\fnt{u}}_f^+$ used to evaluate the numerical flux $\fnt{f}_i^*$ corresponds either to interface values on a neighboring element or an exterior state used to enforce boundary conditions.  Assuming continuity in time, the formulation (\ref{eq:esdg}) is entropy conservative in the sense that
\[
\fnt{1}^T\fnt{M}\td{S(\fnt{u}_h)}{t} + \sum_{i=1}^d\fnt{1}^T\fnt{B}_i\LRp{\fnt{v}_f^T\fnt{f}_i^* - \psi_i(\tilde{\fnt{u}}_f)} = 0.
\]
The proof uses similar techniques as proofs in other papers \cite{chen2017entropy, crean2018entropy, chan2017discretely, chan2019skew}
The proof of entropy conservation relies mainly on Lemma~\ref{lemma:Qhprops}, which states that $\fnt{Q}_{i,h}$ is conservative (e.g., exact for constants) and satisfies the SBP property.  Extensions to the non-conforming setting will utilize the same properties.  

All entropy conservative schemes described here can be made entropy stable by introducing entropy dissipation through mechanisms such as physical or artificial viscosity \cite{tadmor2006entropy,upperman2019entropy}.  We introduce dissipation by incorporating a penalization term into the interface flux \cite{winters2017uniquely}.  For example, Lax-Friedrichs dissipation can be added by modifying the interface flux term $\fnt{B}_i\fnt{f}^*_i$ 
\[
\fnt{B}_i\fnt{f}^*_i \Longrightarrow \fnt{B}_i\fnt{f}^*_i - \frac{\lambda}{2} \jump{\tilde{\fnt{u}}_f}, \qquad  \jump{\tilde{\fnt{u}}_f} = \tilde{\fnt{u}}_f^+ - \tilde{\fnt{u}}_f,
\]
where $\lambda$ is an estimate of the maximum wave speed \cite{chen2017entropy, chan2017discretely}.  

For the remainder of this work, we will construct formulations and prove they are entropy conservative, with the understanding that they can be made entropy stable by incorporating an entropy dissipative penalty term.

\section{Non-conforming meshes}
\label{sec:2}

Section~\ref{sec:1} describes the construction of entropy-stable schemes on geometrically conforming meshes, where each element shares at most one neighbor across a face.  We extend the construction of stable schemes to meshes containing geometric non-conformity, where an element can share a face with two or more neighboring elements.  To ensure stability, the coupling conditions imposed at this non-conforming face must be handled appropriately.  For DG discretizations, this is most naturally achieved by combining composite quadrature rules on non-conforming faces with appropriate evaluations of volume terms \cite{kozdon2018energy}.  However, for entropy-stable schemes using hybridized SBP operators, the naive use of composite quadrature at non-conforming interfaces can significantly increase the computational cost.

For quad and hex elements under tensor product volume quadrature, the differentiation matrices are Kronecker products of 1D differentiation matrices and diagonal mass matrices, which result in sparse operators.  As a result, flux evaluations are only required between ``lines'' of volume nodes \cite{carpenter2014entropy, chan2018efficient}.  Flux evaluations also follow the sparsity pattern of the matrix $\fnt{E}$, which maps from volume nodes to surface nodes.  For Lobatto or Gauss collocation methods on conforming quadrilateral and hexahedral meshes, $\fnt{E}$ is sparse if the surface quadrature nodes are aligned with volume quadrature nodes \cite{chan2018efficient}.  In such cases, flux evaluations are required only between lines of volume nodes and adjacent surface nodes, as shown in Figure~\ref{subfig:aligned}.

Composite quadrature rules on non-conforming meshes, however, are not aligned with volume nodes.  Suppose that $D^k$ is a quadrilateral element with two neighbors across each face, such that every face is non-conforming and utilizes a composite quadrature rule.  The entropy conservative formulation (\ref{eq:esdg}) can be extended to the non-conforming case by redefining the interpolation matrix $\fnt{E}$ as the matrix which interpolates from volume nodes to composite surface nodes.  However, this version of $\fnt{E}$ is fully dense, and evaluating the formulation (\ref{eq:esdg}) requires flux evaluations between each surface node and \textit{all} volume nodes, as illustrated in Figure~\ref{subfig:nonaligned}.  This greatly increases computational costs, especially at high orders of approximation.  

\begin{figure}
\centering
\subfloat[Flux evaluations required for aligned surface nodes]{\raisebox{.1em}{\includegraphics[width=.375\textwidth]{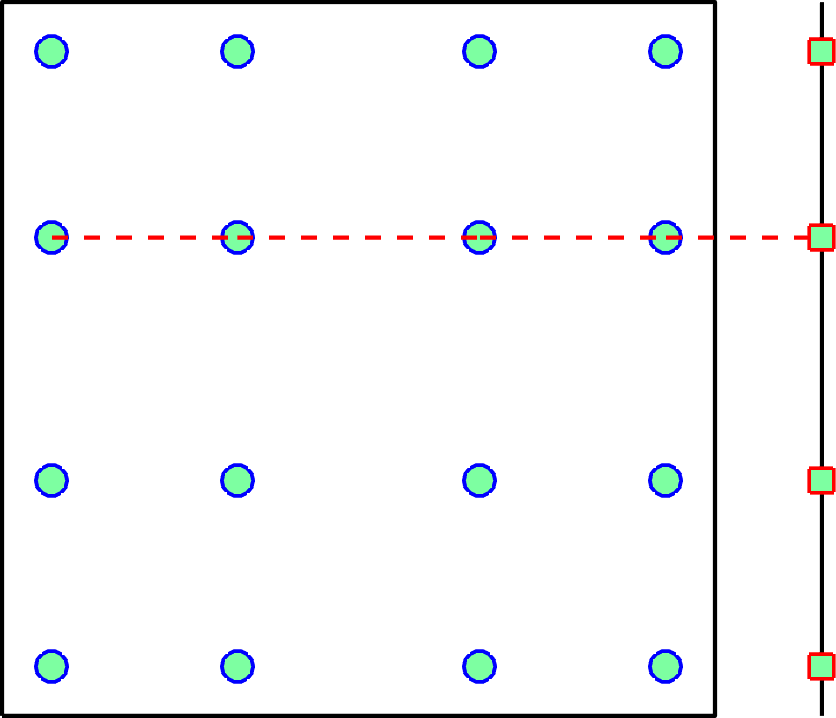}}\label{subfig:aligned}}
\hspace{4em}
\subfloat[Flux evaluations required for non-aligned surface nodes]{\includegraphics[width=.37\textwidth]{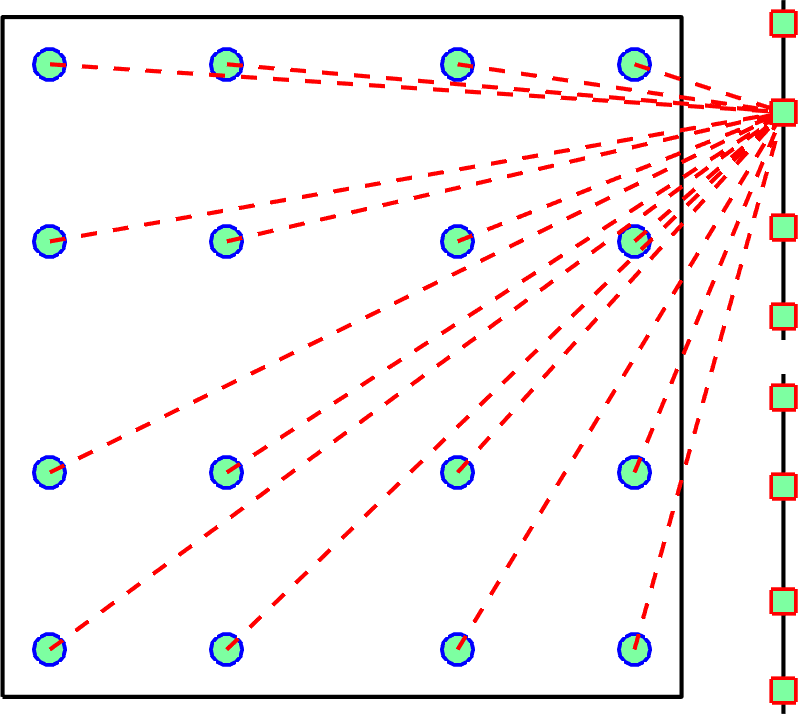}\label{subfig:nonaligned}}
\caption{Nodes between which flux evaluations are required for Gauss nodes. Aligned surface nodes (conforming interfaces) require evaluations between each surface node and a line of volume nodes, while non-aligned surface nodes (non-conforming interfaces) require flux evaluations between a surface node and \textit{all} volume nodes.  } 
\label{fig:fluxsparsity}
\end{figure}

%

%

\section{Entropy-stable mortar formulations}
\label{sec:3}

The goal of this work is to reduce computational costs for Gauss collocation schemes in the presence of non-conforming interfaces.  This can be done by treating composite quadrature nodes as a layer of ``mortar'' nodes which are coupled directly to surface nodes, but not directly to the volume nodes, as illustrated in Figure~\ref{fig:gqcon_noncon}.  This results in modifications of the matrices involved in the entropy-stable formulation (\ref{eq:esdg}).  These modifications preserve both high-order accuracy and entropy stablility, and yield an implementation which is identical to that of (\ref{eq:esdg}) except for a face-local correction to the numerical flux.  
\begin{figure}
\centering
\includegraphics[width=.6\textwidth]{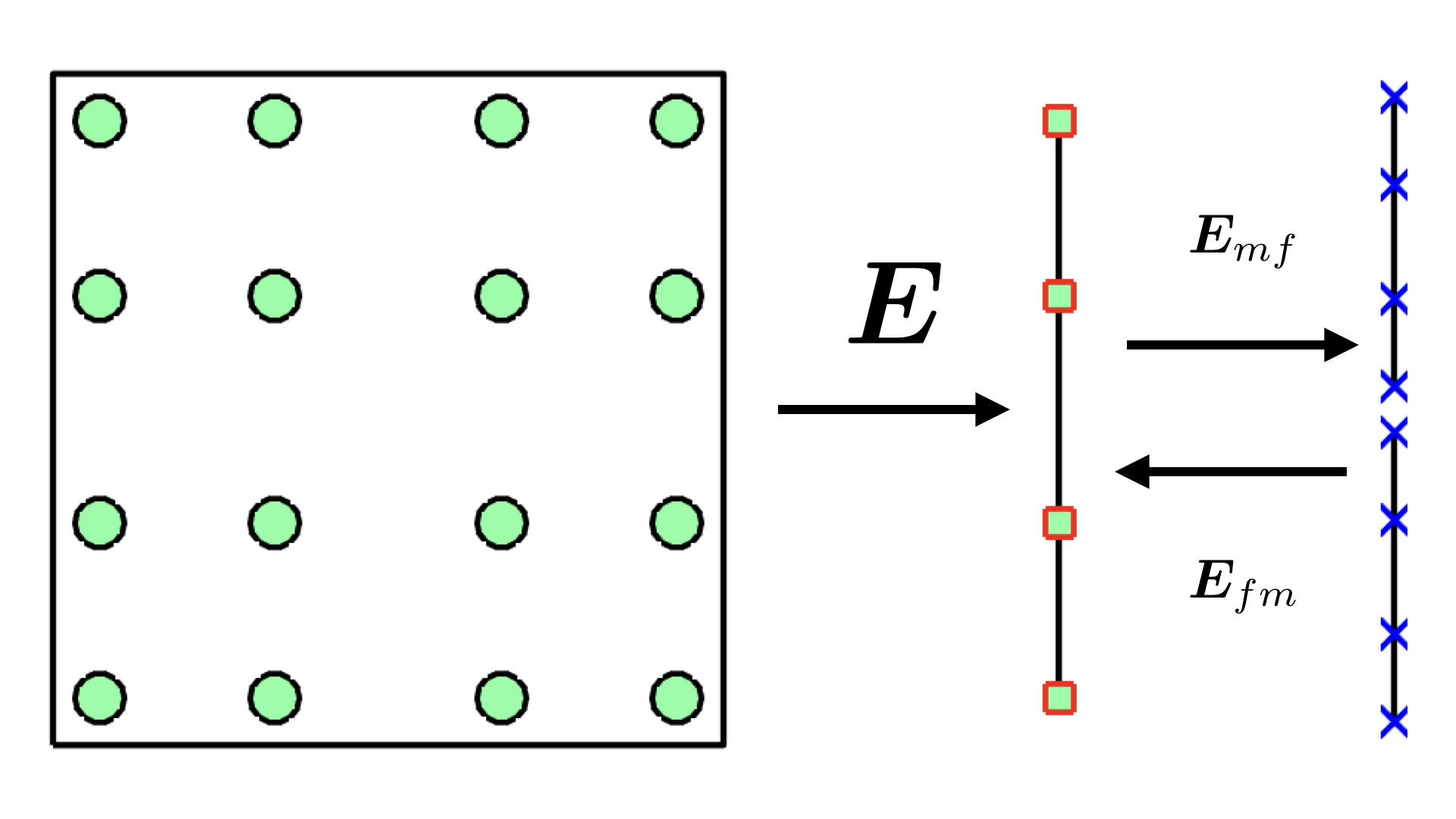}
\caption{Illustration of mortar operators for a Gauss collocation scheme.  The matrix $\fnt{E}$ maps from volume quadrature points to surface quadrature points, $\fnt{E}_{mf}$ maps from surface to mortar surface points, and $\fnt{E}_{fm}$ maps from mortar surface points to surface points.}  
\label{fig:gqcon_noncon}
\end{figure}

We introduce new matrices on the reference element $\hat{D}$.  Let $\hat{\fnt{x}}$ denote volume collocation points, and let $\hat{\fnt{x}}_f$ denote face (surface) points on the boundary $\partial \hat{D}$.  To simplify notation, we assume from this point onwards that the nodes are ordered face-by-face.  In 2D, this implies that $\hat{\fnt{x}}_f$ is 
\begin{equation}
\hat{\fnt{x}}_f = \begin{bmatrix}
\hat{\fnt{x}}_{f,1} &
\hat{\fnt{x}}_{f,2} &
\hat{\fnt{x}}_{f,3} &
\hat{\fnt{x}}_{f,4}
\end{bmatrix}^T,
\label{eq:facenodeordering}
\end{equation}
where $\hat{\fnt{x}}_{f,j}$ denotes the vector of 1D nodal positions on the $j$th face of the reference quadrilateral $[-1,1]^2$.  Recall that the matrix $\fnt{E}$ as defined in (\ref{eq:EBdefs}) interpolates from volume collocation points $\hat{\fnt{x}}$ to a specific ordering of the surface points $\hat{\fnt{x}}_f$.  The ordering (\ref{eq:facenodeordering}) simply corresponds to a permutation of the interpolation matrix $\fnt{E}$.  

We now introduce a second set of mortar points $\hat{\fnt{x}}_m$ on a face of the reference element.  We assume these points also correspond to quadrature nodes with corresponding weights $\fnt{w}_m$.  For example, mortar nodes can be composite Gauss or Lobatto quadrature nodes (for $h$ non-conforming interfaces), higher degree Gauss or Lobatto nodes (for $p$ non-conforming interfaces), or identical to the surface nodes $\hat{\fnt{x}}_f$ (non-mortar interface).   For simplicity of notation, we assume that each face has the same set of mortar nodes; however, it is straightforward to extend this to the case when the mortar nodes vary face-by-face.  

We define the interpolation matrix $\fnt{E}_{mf}$ as the operator which maps values at surface nodes to values at mortar nodes.  In 2D, since the face of a quadrilateral is a 1D line, we can define $\fnt{E}_{mf}$ as the block matrix acting on the surface nodes on the four faces 
\[
\fnt{E}_{mf} = \fnt{I}_{4}\otimes \fnt{E}_m,
\]
where $\fnt{I}_{4}$ is the 4-by-4 identity matrix and $\fnt{E}_m$ is the mortar interpolation matrix over the reference face (interval) $[-1,1]$.  The matrix $\fnt{E}_m$ is defined in terms of $\hat{x}_{m,i}$, the mortar nodes mapped to the reference interval $[-1,1]$,   
\[
(\fnt{E}_{m})_{ij} = \ell_j(\hat{x}_{m,i}), \qquad 1\leq i \leq \text{num.\ mortar points}, \quad 1\leq j \leq (N+1).
\]
We can define an analogous matrix $\fnt{E}_{fm}$ which maps from the mortar nodes back to the surface nodes.  We first define the face mass matrix as the block diagonal matrix whose blocks are 1D diagonal reference mass matrices
\[
\fnt{M}_f = \fnt{I}_{4\times 4} \otimes \hat{\fnt{M}}_{\rm 1D}.
\]
These matrices are defined in 2D for simplicity, but are straightforward to extend to 3D.  
The matrix $\fnt{E}_{fm}$ can now be defined through a quadrature-based $L^2$ projection 
\begin{equation}
\fnt{E}_{fm} = \hat{\fnt{M}}_f^{-1}\fnt{E}_{mf}^T\hat{\fnt{M}}_m, \qquad \hat{\fnt{M}}_m = \fnt{I}_{4}\otimes \diag{\fnt{w}_m}.
\label{eq:Efm}
\end{equation}
Finally, we introduce diagonal boundary matrices on surface and mortar nodes 
\begin{equation}
\hat{\fnt{B}}_{i,f} = \diag{\hat{\fnt{n}}_{i,f}}\hat{\fnt{M}}_f 
\qquad \hat{\fnt{B}}_{i,m} = \diag{\hat{\fnt{n}}_{i,m}}\hat{\fnt{M}}_m,
\label{eq:Bi}
\end{equation}
where $\hat{\fnt{n}}_{i,f}, \hat{\fnt{n}}_{i,m}$ are vectors containing components of the scaled reference normals at face and mortar points, respectively.  


Note that $\fnt{E}_{fm}$ exactly recovers polynomials of a certain degree on the reference face, where the degree of the polynomial is related to the accuracy of the surface and mortar quadratures.
\begin{lemma}
\label{lemma:Efm}
Suppose the face (surface) quadrature is exact for polynomials of degree $N+N_f$ and mortar quadratures are exact for degree $N+N_m$ polynomials.  Then, $\fnt{E}_{fm}$ exactly recovers polynomials of degree $\min(N,N_f,N_m)$.
\end{lemma}
\begin{proof}
Let $u(\bm{x})$ be a polynomial of degree $\min(N,N_f,N_m)$ or less, and let $\fnt{u}_f$ be its values on the face (surface) nodes.  Then, $\fnt{u}_m = \fnt{E}_{mf}\fnt{u}_f$ are the interpolated values of the polynomial on the mortar nodes.  Applying $\fnt{E}_{fm}$ yields
\[
\fnt{E}_{fm}\fnt{u}_m = \hat{\fnt{M}}_f^{-1}\fnt{E}_{mf}^T\hat{\fnt{M}}_m\fnt{E}_{mf}\fnt{u}_f.
\]
The entries of $\fnt{E}_{mf}^T\hat{\fnt{M}}_m\fnt{E}_{mf}$ are integrals of products of Lagrange basis functions with $u(\bm{x})$, where all integrals are approximated using mortar quadrature.  

Since $u$ is degree $\min(N,N_f,N_m)$ and each Lagrange basis function is degree $N$, the integrand is computed exactly under the mortar quadrature.  Moreover, since $u$ is degree $\min(N,N_f,N_m)$, it is also computed exactly using the face (surface) quadrature.  Thus, $\fnt{E}_{mf}^T\hat{\fnt{M}}_m\fnt{E}_{mf} \fnt{u}_f = \hat{\fnt{M}}_f \fnt{u}_f$ by exactness of the face and mortar quadratures, and $\fnt{E}_{fm}\fnt{u}_m = \fnt{u}_f$.
\qed\end{proof}

\subsection{Mortar-based hybridized SBP operators}

The matrices $\fnt{E}_{mf}, \fnt{E}_{fm}, \hat{\fnt{B}}_{i,f}, \hat{\fnt{B}}_{i,m}$ can now be used to construct SBP operators which involve mortar nodes.  We define the mortar-based hybridized SBP operator on the reference element $\hat{\fnt{Q}}_{i,m}$ to be
\begin{align}
\hat{\fnt{Q}}_{i,m} = \frac{1}{2}\begin{bmatrix}
\hat{\fnt{Q}}_i - \hat{\fnt{Q}}_i^T & \fnt{E}^T\hat{\fnt{B}}_{i,f} & \\
-\hat{\fnt{B}}_{i,f}\fnt{E} & & \hat{\fnt{B}}_{i,f} \fnt{E}_{fm}\\
& -\hat{\fnt{B}}_{i,m} \fnt{E}_{mf} & \hat{\fnt{B}}_{i,m}
\end{bmatrix}.
\label{eq:mhsbp}
\end{align}

Since it is not clear at first glance how the hybridized operator $\hat{\fnt{Q}}_{i,h}$ or its mortar-based variant $\hat{\fnt{Q}}_{i,m}$ can be used to perform differentiation, we review intuitive explanations of each operator.  

Let $\fnt{u}$ denote basis coefficients for some function $u(\bm{x})$, and let $f, g$ denote two functions on the reference element.  Suppose $\fnt{u}$ satisfies the following matrix system involving the standard hybridized operator $\hat{\fnt{Q}}_{i,h}$
\begin{align}
\hat{\fnt{M}}\fnt{u} = \begin{bmatrix}
\fnt{I}\\
\fnt{E}
\end{bmatrix}^T \diag{\fnt{f}}\hat{\fnt{Q}}_{i,h} {\fnt{g}}, \qquad \fnt{f} = \begin{bmatrix} f(\hat{\fnt{x}}) \\ f(\hat{\fnt{x}}_f)\end{bmatrix}, \qquad \fnt{g} = \begin{bmatrix} g(\hat{\fnt{x}}) \\ g(\hat{\fnt{x}}_f)\end{bmatrix}.
\label{eq:Qh}
\end{align}
It was shown in \cite{chan2017discretely, chan2019skew} that this corresponds to a high-order accurate approximation of $f\pd{g}{x_i}$.
Expanding out the different blocks of (\ref{eq:Qh}) and using generalized summation by parts for $\hat{\fnt{Q}}_i$ yields
\begin{align*}
\hat{\fnt{M}}\fnt{u} &= 
\diag{f(\hat{\fnt{x}})}\hat{\fnt{Q}}_i g(\hat{\fnt{x}}) + \frac{1}{2}\LRp{\diag{f(\hat{\fnt{x}})}\fnt{E}^T + \fnt{E}^T f(\hat{\fnt{x}}_f)}\hat{\fnt{B}}_i \LRp{g(\hat{\fnt{x}}_f) - \fnt{E} g(\hat{\fnt{x}})}.
\end{align*}
The latter expression can be interpreted as a boundary correction term.  Since $\fnt{E}$ is a high-order accurate boundary interpolation operator, this correction term vanishes if $g$ is a degree $N$ polynomial.  

The mortar-based hybridized SBP operator can be similarly interpreted as a high-order differentiation operator.  Recall that $\fnt{E}_{mf}$ interpolates from face nodes to mortar nodes, and that $\fnt{E}$ interpolates from volume nodes to face nodes.  Thus, the matrix $\fnt{E}_{mf}\fnt{E}$ interpolates from volume nodes to mortar nodes.  We can use this matrix to replicate (\ref{eq:Qh}) for the mortar-based operator $\hat{\fnt{Q}}_{i,m}$.  Let $\fnt{u}$ solve the following system
\begin{align}
\hat{\fnt{M}}\fnt{u} = \begin{bmatrix}
\fnt{I}\\
\fnt{E}\\
\fnt{E}_{mf}\fnt{E}
\end{bmatrix}^T \diag{\fnt{f}}\hat{\fnt{Q}}_{i,m} \diag{\fnt{g}}, \qquad \fnt{f} = \begin{bmatrix} f(\hat{\fnt{x}}) \\ f(\hat{\fnt{x}}_f)\\ f(\hat{\fnt{x}}_m)\end{bmatrix}, \quad \fnt{g} = \begin{bmatrix} g(\hat{\fnt{x}}) \\ g(\hat{\fnt{x}}_f)\\ g(\hat{\fnt{x}}_m)\end{bmatrix}.
\label{eq:Qm}
\end{align}
Expanding out terms yields a similar expression involving multiple correction terms
\begin{align}
\hat{\fnt{M}}\fnt{u} &= \diag{f(\hat{\fnt{x}})}\hat{\fnt{Q}}_i{g}(\hat{\fnt{x}}) \label{eq:mhsbp_correction}\\
&+ \frac{1}{2} \diag{f(\hat{\fnt{x}})}\fnt{E}^T\hat{\fnt{B}}_{i,f}\LRp{g(\hat{\fnt{x}}_f) - \fnt{E}g(\hat{\fnt{x}})} \nonumber\\
&+ \frac{1}{2} \fnt{E}^T\diag{f(\hat{\fnt{x}}_f)}\hat{\fnt{B}}_{i,f}\LRp{\fnt{E}_{fm}g(\hat{\fnt{x}}_m) - \fnt{E}g(\hat{\fnt{x}})} \nonumber\\
&+ \frac{1}{2} \fnt{E}^T\fnt{E}_{mf}^T\diag{f(\hat{\fnt{x}}_m)}\hat{\fnt{B}}_{i,m}\LRp{g(\hat{\fnt{x}}_m) - \fnt{E}_{mf}g(\hat{\fnt{x}}_f)} \nonumber
\end{align}
We can show that these correction terms vanish for polynomial functions $g$.  

\begin{lemma}
\label{lemma:accuracy}
Let $f(\bm{x}) = 1$ and suppose that the face (surface) quadrature is exact for polynomials of degree $N+N_f$ and mortar quadratures are exact for degree $N+N_m$ polynomials.  Then, derivative approximations given by the equations (\ref{eq:Qm}) and (\ref{eq:mhsbp_correction}) are exact if $g(\bm{x})$ is a degree $\min(N,N_f,N_m)$ polynomial.
\end{lemma}
\begin{proof}
Following \cite{chan2017discretely}, if $f = 1$  and $g(\bm{x})$ is a degree $N$ polynomial, then  (\ref{eq:mhsbp_correction}) is exact if the correction terms vanish.  The first and third correction terms are zero by the fact that $\fnt{E}, \fnt{E}_{mf}$ are degree $N$ interpolation operators, and the second correction term vanishes since Lemma~\ref{lemma:Efm} implies that $\fnt{E}_{fm}$ exactly recovers polynomials of degree $\min(N,N_f,N_m)$.
\qed\end{proof}

Finally, we note that the mortar-based hybridized SBP operators satisfy a summation by parts property.
\begin{lemma}
\label{lemma:mhsbp}
Let $\hat{\fnt{Q}}_{i,m}$ be defined as in (\ref{eq:mhsbp}).  Then, 
\[
\hat{\fnt{Q}}_{i,m} + \hat{\fnt{Q}}_{i,m}^T = \begin{bmatrix}
\fnt{0} && \\
&\fnt{0} &\\
&& \hat{\fnt{B}}_{i,m}\end{bmatrix}, \qquad \hat{\fnt{Q}}_{i,m}\fnt{1} = \fnt{0}.
\]
\end{lemma}
\begin{proof}
The SBP property holds if 
\begin{equation}
\hat{\fnt{B}}_{i,f} \fnt{E}_{fm} = (\hat{\fnt{B}}_{i,m} \fnt{E}_{mf})^T = \fnt{E}_{mf}^T\hat{\fnt{B}}_{i,m},  
\label{eq:mhsbpprop}
\end{equation}
where we have used that $\hat{\fnt{B}}_{i,m}$ is diagonal.  Recall that $\fnt{E}_{fm} = \fnt{M}_f^{-1}\fnt{E}_{mf}^T\fnt{M}_m$.  
Then, by the definition of $\hat{\fnt{B}}_{i,m}, \hat{\fnt{B}}_{i,h}$ in (\ref{eq:Bi}), we have that
\begin{align*}
\hat{\fnt{B}}_{i,f} \fnt{E}_{fm} = \diag{\hat{\fnt{n}}_{i,f}}\fnt{M}_f \fnt{M}_f^{-1} \fnt{E}_{mf}^T\fnt{M}_m 
= \diag{\hat{\fnt{n}}_{i,f}} \fnt{E}_{mf}^T\fnt{M}_m.
\end{align*}
Then, since the scaled outward normals $\hat{n}_i\hat{J}_f$  are constant over each face of the reference element and $\fnt{E}_{mf}, \fnt{M}_m$ are block diagonal matrices (with each block corresponding to a face), $\diag{\hat{\fnt{n}}_{i,f}}$ commutes and 
\[
\diag{\hat{\fnt{n}}_{i,f}} \fnt{E}_{mf}^T\fnt{M}_m =\fnt{E}_{mf}^T\fnt{M}_m \diag{\hat{\fnt{n}}_{i,m}}= \fnt{E}_{mf}^T\hat{\fnt{B}}_{i,m}.
\]  
\qed\end{proof}

%

\subsection{Mortar-based SBP operators and entropy conservative formulations on mapped elements}

We can extend mortar-based SBP operators to mapped curvilinear elements by constructing physical operators in a manner akin to (\ref{eq:curvedQ}).  Abusing notation, we use $\fnt{g}_{ij}$ to now denote the vector containing values of $g_{ij} = J\pd{\hat{x}_j}{x_i}$ at volume, surface, and mortar points.  Then, we define physical operators $\fnt{Q}_{i,m}$ via
\begin{equation}
\fnt{Q}_{i,m} = \frac{1}{2}\sum_{j=1}^d \diag{\fnt{g}_{ij}}\hat{\fnt{Q}}_{j,m} + \hat{\fnt{Q}}_{j,m} \diag{\fnt{g}_{ij}}.
\label{eq:curvedQ_mortar}
\end{equation}
Proofs of entropy conservation (e.g., for the formulation (\ref{eq:esdg})) require that the hybridized SBP operators are conservative (e.g., exact for constants) and satisfy the SBP property.  Entropy conservation for mortar-based formulations on mapped elements will require similar properties.  However, in contrast to hybridized SBP operators, conservation for mapped mortar-based operators requires additional constraints on the accuracy of surface and mortar quadratures relative to the polynomial degree of the geometric mapping.  

We assume now that each geometric mapping from reference element $\hat{D}$ to physical element $D^k$ is a polynomial of degree $N_{\rm geo} \leq N$.  We also introduce the space of tensor product polynomials in $d$ dimensions $Q^{N_1,\ldots, N_d}$ as
\[
Q^{N_1,\ldots, N_d} = \LRc{ \hat{x}_1^{i_1}\hat{x}_2^{i_2}\ldots\hat{x}_d^{i_d}, \qquad 0\leq i_k \leq N_k, \quad k = 1,\ldots,d.}
\]
We denote the isotropic tensor product space $Q^N= Q^{N,\ldots,N}$ for conciseness.  

In \cite{chan2019skew}, it was shown that the geometric terms $g_{ij}$ are tensor product polynomials of specific degrees.  For mapped quadrilateral elements, $g_{ij}$ satisfy
\begin{align*}
&g_{i1} \in Q^{N_{\rm geo},N_{\rm geo}-1}\\
&g_{i2} \in Q^{N_{\rm geo}-1,N_{\rm geo}}.
\end{align*}
for $i = 1,2$, and the geometric terms naturally satisfy the GCL condition (\ref{eq:dgcl}).  

For hexahedral elements, it is more challenging to construct geometric terms $g_{ij}$ which satisfy the GCL while retaining high-order accuracy.  This is further complicated by the fact that we must take into account the polynomial degrees of $g_{ij}$ when proving conservation for mortar-based SBP operators.  We consider two approaches for computing hexahedral geometric terms in this work:
\begin{enumerate}
\item Approach 1: the construction of $g_{ij}$ from \cite{kopriva2006metric}, which yields 
\begin{align*}
&g_{ij} \in Q^{N_{\rm geo}}, \qquad i,j = 1,2,3.
\end{align*}
\item Approach 2: the construction of $g_{ij}$ from \cite{kozdon2018energy} (Appendix C.3, see also Footnote 3 in \cite{chan2019skew}), which yields
\begin{align*}
&g_{i1} \in Q^{N_{\rm geo},N_{\rm geo}-1,N_{\rm geo}-1}\\
&g_{i2} \in Q^{N_{\rm geo}-1,N_{\rm geo},N_{\rm geo}-1}\\
&g_{i3} \in Q^{N_{\rm geo}-1,N_{\rm geo}-1,N_{\rm geo}}, \qquad i = 1,2,3.
\end{align*}
\end{enumerate}
These approaches are described in more detail in Appendix~\ref{app:A}.  
\begin{lemma}
\label{lemma:Qmprops_3d}
Suppose $D^k$ is a tensor product element element with tensor product surface quadratures and that $\fnt{Q}_{i,m}$ is constructed using (\ref{eq:curvedQ_mortar}).  Suppose also that the surface quadrature is exact for $Q^{N+N_f}$,  the mortar quadrature is exact for $Q^{N+N_m}$.  If 
\begin{enumerate}
\item $D^k$ is a quadrilateral element and $N_{\rm geo} \leq \min(N,N_f+1,N_m+1)$, or
\item $D^k$ is a hexahedral element, $g_{ij}$ is constructed via Approach 1 \cite{kopriva2006metric}, and $N_{\rm geo} \leq \min(N,N_f,N_m)$, or
\item $D^k$ is hexahedral element, $g_{ij}$ is constructed via Approach 2 \cite{kozdon2018energy}, and $N_{\rm geo} \leq \min(N,N_f+1,N_m+1)$, 
\end{enumerate}
then the following properties hold:
\begin{align*}
\fnt{Q}_{i,m} + \fnt{Q}_{i,m}^T &= \begin{bmatrix}
\fnt{0} & &\\
& \fnt{0} &\\
& & \fnt{B}_{i,m} \end{bmatrix}, \quad \text{(SBP property)}\\
\fnt{Q}_{i,m}\fnt{1} &= \fnt{0}, \quad \text{(conservation)},
\end{align*}

where the boundary matrix 
\begin{equation}
\fnt{B}_{i,m} = \diag{\fnt{n}_{i,m}\circ \fnt{w}_m}
\end{equation}
is the diagonal matrix whose entries consist of the scaled physical normals. 
\end{lemma}
\begin{proof}
Using the SBP property of $\hat{\fnt{Q}}_{i,m}$, proving the SBP property follows the same steps as the proof of Lemma~\ref{lemma:Qhprops}.  
Expanding out $\fnt{Q}_{i,m} + \fnt{Q}_{i,m}^T$ yields
\[
\frac{1}{2}\sum_{j=1}^d \diag{\fnt{g}_{ij}}\hat{\fnt{Q}}_{j,m} + \hat{\fnt{Q}}_{j,m}^T\diag{\fnt{g}_{ij}} + \hat{\fnt{Q}}_{j,m} \diag{\fnt{g}_{ij}} +  \diag{\fnt{g}_{ij}}\hat{\fnt{Q}}_{j,m}^T.
\]
For each term in the sum, we use the SBP property of $\hat{\fnt{Q}}_{j,m}$.  The volume terms cancel, leaving only surface terms
\[
\frac{1}{2}\sum_{j=1}^d \LRp{\begin{bmatrix}
\fnt{0} & &\\
& \fnt{0} &\\
& & \hat{\fnt{B}}_{j,m} \end{bmatrix}\diag{\fnt{g}_{ij}} +  \diag{\fnt{g}_{ij}}\begin{bmatrix}
\fnt{0} & &\\
& \fnt{0} &\\
& & \hat{\fnt{B}}_{j,m} \end{bmatrix}} = \begin{bmatrix}
\fnt{0} & &\\
& \fnt{0} &\\
& & \fnt{B}_{i,m} \end{bmatrix}.
\]
where we have used (\ref{eq:nJ_Gnhat}) and that $\hat{\fnt{B}}_{j,m}$ and $\diag{\fnt{g}_{ij}}$ are diagonal matrices.  


To show conservation, expanding out $\fnt{Q}_{i,m}\fnt{1} $ yields
\[
\fnt{Q}_{i,m}\fnt{1} = \frac{1}{2}\sum_{j=1}^d \diag{\fnt{g}_{ij}}\hat{\fnt{Q}}_{j,m}\fnt{1} + \hat{\fnt{Q}}_{j,m} \diag{\fnt{g}_{ij}}\fnt{1} = \frac{1}{2}\sum_{j=1}^d\hat{\fnt{Q}}_{j,m} \fnt{g}_{ij}
\]
since $\hat{\fnt{Q}}_{j,m}\fnt{1} = \fnt{0}$ by Lemma~\ref{lemma:mhsbp}.  Expanding out remaining terms using (\ref{eq:mhsbp_correction}) yields
\begin{align*}
\sum_{j=1}^d \hat{\fnt{Q}}_j g_{ij}(\hat{\fnt{x}}) &+ \frac{1}{2} \fnt{E}^T\hat{\fnt{B}}_{j,f}\LRp{g_{ij}(\hat{\fnt{x}}_f) - \fnt{E}g_{ij}(\hat{\fnt{x}})} \\
&+ \frac{1}{2} \fnt{E}^T\hat{\fnt{B}}_{j,f}\LRp{\fnt{E}_{fm}g_{ij}(\hat{\fnt{x}}_m) - \fnt{E}g_{ij}(\hat{\fnt{x}})} \\
&+ \frac{1}{2} \fnt{E}^T\fnt{E}_{mf}^T\hat{\fnt{B}}_{j,m}\LRp{g_{ij}(\hat{\fnt{x}}_m) - \fnt{E}_{mf}g_{ij}(\hat{\fnt{x}}_f)}.
\end{align*}
where $g_{ij}(\hat{\fnt{x}}), g_{ij}(\hat{\fnt{x}}_f), g_{ij}(\hat{\fnt{x}}_m)$ denote the values of $g_{ij}$ at volume, surface, and mortar points.  By the fact that $g_{ij} \in Q^N$ and $\fnt{E}, \fnt{E}_{mf}$ are degree $N$ interpolation operators, the first and third correction terms vanish.  The remaining boundary correction terms are
\begin{equation}
\fnt{E}^T\hat{\fnt{B}}_{j,f}\LRp{\fnt{E}_{fm}g_{ij}(\hat{\fnt{x}}_m) - \fnt{E}g_{ij}(\hat{\fnt{x}})}.  
\label{eq:bcor}
\end{equation}
Recall that $\hat{\fnt{B}}_{i,f}$ is a diagonal matrix whose entries are $\hat{\fnt{n}}_{i,m}$ scaled by the mortar quadrature weights.  On tensor product elements, $\hat{\fnt{n}}_{i,m}=\pm 1$ on faces where $\hat{x}_i = \pm 1$ and zero otherwise.  Thus, restricting $\hat{x}_1 = \pm 1$ yields the boundary values of $g_{11}$ on faces where the correction term (\ref{eq:bcor}) is non-zero.  We can now show that (\ref{eq:bcor}) vanishes for $i = j = 1$ (the cases of $i = 2,\ldots,d$ are similar).  
\begin{enumerate}
\item If $D^k$ is a quadrilateral element, $g_{11} \in Q^{N_{\rm geo},N_{\rm geo}-1}$, and the surface traces of $g_{11}$ are in the 1D trace space $P^{N_{\rm geo}-1}$.  
\item If $D^k$ is a hexahedral element and $g_{ij}$ is constructed using Approach 2 \cite{kozdon2018energy}, $g_{11} \in Q^{N_{\rm geo}, N_{\rm geo}-1,N_{\rm geo}-1}$, and the surface traces of $g_{11}$ are in the quadrilateral trace space $Q^{N_{\rm geo}-1,N_{\rm geo}-1}$.  
\end{enumerate}
Since $N_{\rm geo} \leq N_m+1$, by Lemma~\ref{lemma:Efm} and the assumption that $N_{\rm geo} \leq \min(N, N_f+1, N_m+1)$, $\fnt{E}_{fm}g_{ij}(\hat{\fnt{x}}_m) = \fnt{E}g_{ij}(\hat{\fnt{x}})$ and (\ref{eq:bcor}) vanishes.  

If Approach 1 \cite{kopriva2006metric} is used, then surface traces of $g_{11}$ are contained in $Q^{N_{\rm geo},N_{\rm geo}}$.  Then, (\ref{eq:bcor}) vanishes since $N_{\rm geo} \leq \min(N,N_f,N_m)$.  In both cases, the remaining terms vanish assuming that $g_{ij}$ satisfies the discrete GCL (\ref{eq:dgcl}).
\qed\end{proof}

An entropy conservative mortar formulation on a mapped element is then 
\begin{gather}
{\fnt{M}}\td{\fnt{u}_h}{t} + \begin{bmatrix} \fnt{I} \\ \fnt{E} \\ \fnt{E}_{mf}\fnt{E} \end{bmatrix}^T
\sum_{i=1}^d \LRp{2{\fnt{Q}}_{i,m} \circ \fnt{F}_i}\fnt{1} + \fnt{E}^T\fnt{E}_{mf}^T{\fnt{B}}_{i,m}\LRp{\fnt{f}_i^*-\bm{f}_i(\tilde{\fnt{u}}_m)} = 0 \label{eq:esdgm}\\
\tilde{\fnt{u}}_m = \bm{u}(\fnt{v}_m), \qquad
\fnt{v}_m = \fnt{E}_{mf}\fnt{v}_f, \qquad
\fnt{f}^*_i = \bm{f}_{i,S}\LRp{\tilde{\fnt{u}}_m^+,\tilde{\fnt{u}}_m},\nonumber
\end{gather}
Again, we assume geometric terms on curved elements are approximated using polynomials, and that the normals are constructed via (\ref{eq:nJ_Gnhat}) and polynomial interpolation.  

\begin{theorem}
Assuming continuity in time, the local formulation (\ref{eq:esdgm}) satisfies 
\begin{equation}
\fnt{1}^T\fnt{M}\td{S(\fnt{u}_h)}{t} + \sum_{i=1}^d\fnt{1}^T\fnt{B}_{i,m}\LRp{\fnt{v}_m^T\fnt{f}_i^* - \psi_i(\tilde{\fnt{u}}_m)} = 0.
\label{eq:localec}
\end{equation}
\end{theorem}
\begin{proof}
The main steps are algebraically very similar to proofs in \cite{chan2018discretely, chan2018efficient, chan2019skew}, with Lemma~\ref{lemma:Qmprops_3d} used in place of Lemma~\ref{lemma:Qhprops}.  We begin by testing (\ref{eq:esdg}) with the vector of entropy variables evaluated at nodal points $\bm{v}(\fnt{u}_h)$.  We apply the chain rule in time and use that $\fnt{M}$ is diagonal to show that 
\[
\bm{v}(\fnt{u}_h)^T\fnt{M}\td{\fnt{u}_h}{t} = \fnt{1}^T\fnt{M}\LRp{\bm{v}(\fnt{u}_h)\circ \td{\fnt{u}_h}{t}} =  \fnt{1}^T\fnt{M}\td{S(\fnt{u}_h)}{t}.
\]
For the volume term, we use the SBP property (Lemma~\ref{lemma:Qmprops_3d}) to arrive at
\begin{align*}
&\bm{v}(\fnt{u}_h)^T\begin{bmatrix} \fnt{I} \\ \fnt{E} \\ \fnt{E}_{mf}\fnt{E} \end{bmatrix}^T\LRp{2\fnt{Q}_{i,m} \circ \fnt{F}_i}\fnt{1} = \\
&\fnt{v}^T\LRp{\LRp{\fnt{Q}_{i,m}-\fnt{Q}_{i,m}^T} \circ \fnt{F}_i}\fnt{1}  + \LRp{\fnt{E}\; \bm{v}(\fnt{u}_h)}^T\LRp{\begin{bmatrix}\fnt{0} && \\ & \fnt{0} & \\ & & \fnt{B}_{i,m}\end{bmatrix} \circ\fnt{F}_i}\fnt{1},
\end{align*}
where, again
\[
\fnt{v} = \begin{bmatrix} \bm{v}(\fnt{u}_h) \\ \fnt{v}_f \\ \fnt{v}_m\end{bmatrix} = 
{\begin{bmatrix} \fnt{I} \\ \fnt{E} \\ \fnt{E}_{mf}\fnt{E} \end{bmatrix}\bm{v}(\fnt{u}_h)}.  
\]
Note that, by the consistency of the entropy conservative flux $\bm{f}_{i,S}$, the diagonal of $\fnt{F}_i$ is the flux evaluated at volume, surface, and mortar points. Thus, the latter term reduces to 
\[
\fnt{v}\LRp{\begin{bmatrix}\fnt{0} && \\ & \fnt{0} & \\ & & \fnt{B}_{i,m}\end{bmatrix} \circ\fnt{F}_i}\fnt{1} = \fnt{v}_m^T \fnt{B}_{i,m}\bm{f}_i(\tilde{\fnt{u}}_m) =  \fnt{1}^T \fnt{B}_{i,m}\LRp{\fnt{v}_m^T\bm{f}_i(\tilde{\fnt{u}}_m)}.
\]
where we have used that $\fnt{B}_{i,m}$ is diagonal.  The final part of the proof uses that 
\begin{align*}
\fnt{v}^T\LRp{\LRp{\fnt{Q}_{i,m}-\fnt{Q}_{i,m}^T} \circ \fnt{F}_i}\fnt{1} &= \sum_{jk} {\fnt{v}}_j \LRp{\fnt{Q}_{i,m}-\fnt{Q}_{i,m}^T}_{jk} \LRp{\fnt{F}_{i}}_{jk}\\
&= \sum_{jk} \LRp{\fnt{Q}_{i,m}}_{jk} \LRp{{\fnt{v}}_j-{\fnt{v}}_k}^T\bm{f}_{i,S}\LRp{\tilde{\fnt{u}}_j,\tilde{\fnt{u}}_k}\\
&= \sum_{jk} \LRp{\fnt{Q}_{i,m}}_{jk} \LRp{\psi_i(\tilde{\fnt{u}}_j)-\psi_i(\tilde{\fnt{u}}_k)} \\
&= \psi_i(\tilde{\fnt{u}})^T\fnt{Q}_{i,m}\fnt{1} - \fnt{1}^T\fnt{Q}_{i,m}\psi_i(\tilde{\fnt{u}})
\end{align*}
By Lemma~\ref{lemma:Qmprops_3d}, $\fnt{Q}_{i,m}\fnt{1} = \fnt{0}$.  The proof is completed by applying the SBP property (\ref{eq:hsbp}) to the remaining term, which reduces to $- \fnt{1}^T\fnt{Q}_{i,m}\psi_i(\tilde{\fnt{u}}) =-\fnt{1}^T\fnt{B}_{i,m}\psi_i(\tilde{\fnt{u}}_m)$.
\qed\end{proof}

Global entropy conservation can be derived by summing up (\ref{eq:localec}) over each element
\[
\sum_k \fnt{1}^T\fnt{M}^k\td{S(\fnt{u}^k_h)}{t} + \sum_k \sum_{i=1}^d \fnt{1}^T\fnt{B}^k_{i,m}\LRp{\LRp{\fnt{v}^k_m}^T\fnt{f}^{k,*}_i - \psi_i(\tilde{\fnt{u}}^k_m)} = 0,
\]
where have introduced the superscript $k$ to denote operators and local solutions on the $k$th element $D^k$.  The first term corresponds to the average entropy over the entire domain.  Assume without loss of generality a periodic domain, such that all interfaces are interior interfaces.  Exchanging interface contributions involving the numerical flux between neighboring elements yields 
\begin{align}
\sum_k &\sum_{i=1}^d \fnt{1}^T\fnt{B}^k_{i,m}\LRp{\LRp{\fnt{v}^k_m}^T\fnt{f}^{k,*}_i - \psi_i(\tilde{\fnt{u}}^k_m)} = \label{eq:ecinterfaces}\\
&\sum_k \sum_{i=1}^d \fnt{1}^T\fnt{B}^k_{i,m}\LRp{\frac{1}{2}\LRp{\fnt{v}^k_m- \fnt{v}^{k,+}_m }^T\fnt{f}^{k,*}_i - \psi_i(\tilde{\fnt{u}}^k_m)}\nonumber
\end{align}
where $\fnt{v}^{k,+}_m$ denotes the exterior value of the entropy variables on the neighboring element.  If the numerical flux is computed using the entropy conservative flux $\fnt{f}^{k,*}_i = \bm{f}_{i,S}\LRp{\tilde{\fnt{u}}^k_m,\tilde{\fnt{u}}^{k,+}_m}$, then 
\[
\LRp{\fnt{v}^k_m- \fnt{v}^{k,+}_m }^T\bm{f}_{i,S}\LRp{\tilde{\fnt{u}}^k_m,\tilde{\fnt{u}}^{k,+}_m} = \psi_i(\tilde{\fnt{u}}^k_m) - \psi_i(\tilde{\fnt{u}}^{k,+}_m).
\]
Returning the contributions $\psi_i(\tilde{\fnt{u}}^{k,+}_m)$ to the neighboring element simplifies (\ref{eq:ecinterfaces}) as follows
\begin{align*}
&\sum_k \sum_{i=1}^d \fnt{1}^T\fnt{B}^k_{i,m}\LRp{\frac{1}{2}\LRp{\fnt{v}^k_m- \fnt{v}^{k,+}_m }^T\fnt{f}^{k,*}_i - \psi_i(\tilde{\fnt{u}}^k_m)}\\
&= \sum_k \sum_{i=1}^d \fnt{1}^T\fnt{B}^k_{i,m}\LRp{\frac{1}{2}\LRp{\psi_i(\tilde{\fnt{u}}^k_m) - \psi_i(\tilde{\fnt{u}}^{k,+}_m) - \psi_i(\tilde{\fnt{u}}^k_m)}}\\
&= \sum_k \sum_{i=1}^d \fnt{1}^T\fnt{B}^k_{i,m}\LRp{\psi_i(\tilde{\fnt{u}}^k_m) - \psi_i(\tilde{\fnt{u}}^k_m)} = 0.
\end{align*}
When the domain is non-periodic, the interface contributions do not vanish at the boundary.   These boundary terms can be made non-negative (e.g., entropy conservative or entropy stable) if boundary conditions are imposed in an appropriate manner \cite{svard2014entropy, chen2017entropy}.

\subsection{Comparison of Lobatto and Gauss nodes: requirements for entropy conservation (stability) and accuracy}
\label{sec:compare}

This section summarizes theoretical differences between entropy conservative (stable) Lobatto and Gauss collocation schemes on non-conforming meshes based on Lemma~\ref{lemma:Qmprops_3d} and Lemma~\ref{lemma:accuracy}.  The main differences are that the use of Lobatto nodes induces additional requirements on the construction of geometric terms for curved meshes and lowers the expected order of accuracy by one degree.  

Entropy-stable high-order collocation schemes on tensor product elements typically utilize tensor product quadrature rules based on $(N+1)$-point Lobatto nodes \cite{carpenter2014entropy, gassner2016split} or $(N+1)$-point Gauss nodes \cite{chan2018efficient}.  On non-conforming meshes, the most natural choice of mortar nodes is a composite Lobatto or Gauss rule.  Recall that an $(N+1)$-point Lobatto rule is exact for polynomials of degree $(2N-1)$, while an $(N+1)$-point Gauss rule is exact for polynomials of degree $(2N+1)$.  

First, we discuss the expected accuracy of each scheme.  We have not conducted a rigorous error analysis; however, Lemma~\ref{lemma:accuracy} implies that a mortar-based SBP operator under Lobatto quadrature is exact for degree $(N-1)$ polynomials, while a mortar-based SBP operator under Gauss quadrature is exact for degree $N$ polynomials.  In other words, the order of accuracy of a Gauss collocation scheme can be up to one degree higher than for a Lobatto collocation scheme.  

We note that order reduction for Lobatto nodes is not consistently evident in numerical experiments, both for conforming \cite{chan2018efficient, hindenlang2019order} and non-conforming meshes \cite{fernandez2019entropy}.  In \cite{friedrich2017entropy}, order reduction with respect to the $L^2$ error was observed using Lobatto nodes on 2D non-conforming meshes.  However, in \cite{fernandez2019entropy}, $L^2$ order reduction is not observed using Lobatto nodes on 3D non-conforming hexahedral meshes, though order reduction is observed for the $L^\infty$ error.  Similar inconsistent behavior is observed for the 2D and 3D experiments in this paper when using Lobatto nodes.  We do note observe order reduction for Gauss nodes in any of our experiments.

Next, we discuss requirements on the construction of geometric terms.  On non-conforming quadrilateral meshes, Theorem~\ref{lemma:Qmprops_3d} implies that both Lobatto and Gauss collocation schemes are stable for $N_{\rm geo}\leq N$, i.e., for isoparametric curved mappings.  On non-conforming hexahedral meshes, Theorem~\ref{lemma:Qmprops_3d} implies that Gauss nodes are also stable for $N_{\rm geo}\leq N$, but that Lobatto nodes are stable only for $N_{\rm geo} \leq (N-1)$ if Approach 1 \cite{kopriva2006metric} is used to compute geometric terms.  

The restriction on $N_{\rm geo} \leq (N-1)$ for Lobatto nodes is not necessary for entropy conservation (stability) on conforming hexahedral meshes.  However, when using Approach 1 to compute geometric terms, the introduction of non-aligned mortar nodes at non-conforming curved interfaces induces this additional restriction.  Moreover, this condition is consistent with Theorem 3 in \cite{fernandez2019entropy}.  One can remove this constraint on the degree of the geometric mapping by using Approach 2 \cite{kozdon2018energy} to compute geometric terms.  If Approach 2 is used, then Lobatto collocation schemes remain entropy conservative (stable) on curved hexahedral meshes for all $N_{\rm geo} \leq N$.  Appendix~\ref{app:A} describes Approach 1 \cite{kopriva2006metric} and Approach 2 \cite{kozdon2018energy} in more detail, and demonstrates numerically that both approaches achieve similar high-order accuracy in approximating geometric terms.

\section{A mortar-based implementation}
\label{sec:mortarimplement}
While the formulation (\ref{eq:esdgm}) is convenient for analysis, it is less convenient to implement.  However, using properties of discretization matrices, we can show that the mortar-based formulation (\ref{eq:esdgm}) is equivalent to a conforming formulation (\ref{eq:esdg}) with a modified numerical flux involving face-local correction terms.  

First, we reformulate volume terms by incorporating geometric terms into the definition of the flux matrices.  Note that $\fnt{Q}_{i,m}$ can alternatively be defined as \cite{gassner2016split}
\[
\fnt{Q}_{i,m} = \sum_{j=1}^d \hat{\fnt{Q}}_{j,m} \circ \avg{\fnt{g}_{ij}}, \qquad \avg{\fnt{g}_{ij}}_{kl} = \frac{1}{2}\LRp{ \LRp{\fnt{g}_{ij}}_k + \LRp{\fnt{g}_{ij}}_l},
\]
where $k,l$ are indices over the total number of volume, surface, and mortar points, and $\fnt{g}_{ij}$ is the vector containing values of $g_{ij}$ at volume, surface, and mortar points.  
The volume term of (\ref{eq:esdgm}) can then be rewritten as
\begin{gather}
\begin{bmatrix} \fnt{I} \\ \fnt{E} \\ \fnt{E}_{mf}\fnt{E} \end{bmatrix}^T\sum_{i=1}^d \LRp{2{\fnt{Q}}_{i,m} \circ \fnt{F}_i}\fnt{1} = 
\begin{bmatrix} \fnt{I} \\ \fnt{E} \\ \fnt{E}_{mf}\fnt{E} \end{bmatrix}^T\sum_{i=1}^d \LRp{2\hat{\fnt{Q}}_{i,m} \circ \hat{\fnt{F}}_i}\fnt{1} \label{eq:Qmreform}\\
\hat{\fnt{F}}_i = \sum_{j=1}^d \avg{\fnt{g}_{ij}}\circ \fnt{F}_i.\nonumber
\end{gather}
We next decompose $\hat{\fnt{F}}_i$ into interactions between volume nodes, surface nodes, and mortar nodes
\[
\hat{\fnt{F}}_i = \begin{bmatrix}
\hat{\fnt{F}}_i^{vv} & \hat{\fnt{F}}_i^{vf} & \hat{\fnt{F}}_i^{vm}\\
\hat{\fnt{F}}_i^{fv} & \hat{\fnt{F}}_i^{ff} & \hat{\fnt{F}}_i^{fm}\\
\hat{\fnt{F}}_i^{mv} & \hat{\fnt{F}}_i^{mf} & \hat{\fnt{F}}_i^{mm}
\end{bmatrix}
\]
By expanding out the definition of the mortar-based SBP operator $\hat{\fnt{Q}}_{i,m}$ and using properties of the matrices $\hat{\fnt{B}}_{i,m}, \hat{\fnt{F}}_i$, (\ref{eq:Qmreform}) can be rewritten as
\begin{align*}
\begin{bmatrix} \fnt{I} \\ \fnt{E} \\ \fnt{E}_{mf}\fnt{E} \end{bmatrix}^T
&\LRp{\begin{bmatrix}
\hat{\fnt{Q}}_i-\hat{\fnt{Q}}_i^T & \fnt{E}^T\hat{\fnt{B}}_i &\\
-\hat{\fnt{B}}_i\fnt{E} &  & \hat{\fnt{B}}_{i}{\fnt{E}}_{fm} \\
& -\hat{\fnt{B}}_{i,m}{\fnt{E}}_{mf} & 
\end{bmatrix} \circ \hat{\fnt{F}}_i}\fnt{1} 
+ \fnt{E}^T\fnt{E}_{mf}^T \LRp{\hat{\fnt{B}}_{i,m} \circ \hat{\fnt{F}}_i^{mm}}\fnt{1} 
\\
=& 
\begin{bmatrix} \fnt{I} \\ \fnt{E} \end{bmatrix}^T
\LRp{\begin{bmatrix}
\hat{\fnt{Q}}_i-\hat{\fnt{Q}}_i^T  & \fnt{E}^T\hat{\fnt{B}}_i\\
-\hat{\fnt{B}}_i\fnt{E} & \hat{\fnt{B}}_i\\
\end{bmatrix} \circ \begin{bmatrix}
\hat{\fnt{F}}_i^{vv} & \hat{\fnt{F}}_i^{vf} \\
\hat{\fnt{F}}_i^{fv} & \hat{\fnt{F}}_i^{ff}
\end{bmatrix} }\fnt{1} \\
&+ \fnt{E}^T\LRp{\fnt{E}_{mf}^T \LRp{\hat{\fnt{B}}_{i,m} \circ \hat{\fnt{F}}_i^{mm}}\fnt{1} - \LRp{\hat{\fnt{B}}_i\circ \hat{\fnt{F}}^{ff}_i}\fnt{1} } \\
&+ \fnt{E}^T \LRp{\LRp{\hat{\fnt{B}}_i\fnt{E}_{fm}}\circ \hat{\fnt{F}}_i^{fm}}\fnt{1} - \fnt{E}^T\fnt{E}_{mf}^T \LRp{ \LRp{\hat{\fnt{B}}_{i,m}\fnt{E}_{mf}} \circ \hat{\fnt{F}}_i^{mf}}\fnt{1}.
\end{align*}
The above expressions can be simplified.  First, evaluating $\hat{\fnt{F}}_i^{mm},\hat{\fnt{F}}_i^{ff}$ yields
\[
\hat{\fnt{F}}_i^{mm} = \diag{\sum_{j=1}^d g_{ij}(\hat{\fnt{x}}_m) \circ \bm{f}_i(\tilde{\fnt{u}}_m)}, 
\quad 
\hat{\fnt{F}}_i^{ff} = \diag{\sum_{j=1}^d g_{ij}(\hat{\fnt{x}}_f) \circ \bm{f}_i(\tilde{\fnt{u}}_f)},
\]
Moreover, recall that the physical normals are defined as products of $g_{ij}$ with reference normals (\ref{eq:nJ_Gnhat}).  Since $\hat{\fnt{B}}_{i,m}, \hat{\fnt{B}}_{i}$ contain scaled normals on the reference element on the diagonal, we can rewrite the correction terms using the physical normal boundary matrices $\fnt{B}_{i,m}, \fnt{B}_i$
\begin{align*}
&\fnt{E}^T\LRp{\fnt{E}_{mf}^T \LRp{\hat{\fnt{B}}_{i,m} \circ \hat{\fnt{F}}_i^{mm}}\fnt{1} - \LRp{\hat{\fnt{B}}_i\circ \hat{\fnt{F}}^{ff}_i}\fnt{1} } \\
&= \fnt{E}^T\LRp{\fnt{E}_{mf}^T \fnt{B}_{i,m}\bm{f}_i(\tilde{\fnt{u}}_m) - \fnt{B}_{i}\bm{f}_i(\tilde{\fnt{u}}_f) }.
\end{align*}
The remaining correction terms can also be simplified by noting that $\fnt{E}_{fm}$ are block diagonal matrices, with each block corresponding to a face, and that the scaled reference normals (i.e., the diagonal entries of $\hat{\fnt{B}}_{i,m}, \hat{\fnt{B}}_{i}$) are constant over each face.  
Then, using (\ref{eq:nJ_Gnhat}), we can rewrite the final set of correction terms as
\begin{align*}
&\fnt{E}^T \LRp{\LRp{\hat{\fnt{B}}_i\fnt{E}_{fm}}\circ \hat{\fnt{F}}_i^{fm}}\fnt{1} - \fnt{E}^T\fnt{E}_{mf}^T \LRp{ \LRp{\hat{\fnt{B}}_{i,m}\fnt{E}_{mf}} \circ \hat{\fnt{F}}_i^{mf}}\fnt{1}\\
& = \fnt{E}^T \diag{\fnt{w}_f}\LRp{{\fnt{E}_{fm}}\circ \fnt{F}_i^{fm} \circ \avg{\fnt{n}_i}_{fm}}\fnt{1} \\
&\qquad - \fnt{E}^T\fnt{E}_{mf}^T \diag{\fnt{w}_m} \LRp{ {\fnt{E}_{mf}} \circ {\fnt{F}}_i^{mf} \circ \avg{\fnt{n}_i}_{mf}}\fnt{1}\\
& = \fnt{E}^T \diag{\fnt{w}_f}\left(\LRp{{\fnt{E}_{fm}}\circ \fnt{F}_i^{fm} \circ \avg{\fnt{n}_i}_{fm}}\fnt{1} \right.\\
&\qquad \left. - \fnt{E}_{fm} \LRp{ {\fnt{E}_{mf}} \circ {\fnt{F}}_i^{mf} \circ \avg{\fnt{n}_i}_{mf}}\fnt{1}\right)
\end{align*}
where $\fnt{F}_i^{fm} = \LRp{\fnt{F}_i^{mf}}^T$ are the flux matrices whose entries are evaluations of $\bm{f}_{i,S}$ between solution values at mortar and surface nodes, and $\avg{\fnt{n}_i}_{fm} = \avg{\fnt{n}_i}_{mf}^T$ are matrices whose entries are arithmetic averages of $n_iJ_f$ (the $i$th component of the scaled normal vector) between each mortar and surface point. Note that we have used that $ \fnt{E}_{fm} = \diag{\fnt{w}_f}^{-1} \fnt{E}_{mf}^T\diag{\fnt{w}_m}$ in the final line.

Adding these correction terms back into the formulation (\ref{eq:esdgm}) yields 
\begin{gather}
{\fnt{M}}\td{\fnt{u}_h}{t} + \begin{bmatrix} \fnt{I} \\ \fnt{E} \end{bmatrix}^T
\sum_{i=1}^d \LRp{2{\fnt{Q}}_{i,h} \circ \fnt{F}_i}\fnt{1} + \fnt{E}^T\LRp{\fnt{f}_i^*- \fnt{B}_i\bm{f}_i(\tilde{\fnt{u}}_f)} = 0 \label{eq:esdgm_mortar}\\
\fnt{f}_i^* = \fnt{E}_{mf}^T{\fnt{B}}_{i,m}\bm{f}_{i,S}(\tilde{\fnt{u}}_m^+,\tilde{\fnt{u}}_m) + \fnt{\delta f}_i,\nonumber\\
\fnt{\delta f}_i = \diag{\fnt{w}_f}\LRp{ \LRp{{\fnt{E}_{fm}}\circ \fnt{F}_i^{fm} \circ \avg{\fnt{n}_i}_{fm}}\fnt{1} - \fnt{E}_{fm} \LRp{ {\fnt{E}_{mf}} \circ {\fnt{F}}_i^{mf} \circ \avg{\fnt{n}_i}_{mf}}\fnt{1}}.\nonumber
\end{gather}  
\begin{remark}
Note that $\fnt{\delta f}_i$ can be rewritten in terms of a Hadamard product 
\begin{align}
\fnt{\delta f}_i = 
\begin{bmatrix}
\fnt{I}\\
\fnt{E}_{mf}
\end{bmatrix}^T
\diag{\begin{bmatrix}
\fnt{w}_f  \\
\fnt{w}_m
\end{bmatrix}}
\LRp{\begin{bmatrix}
& \fnt{E}_{fm}\\
-\fnt{E}_{mf} &
\end{bmatrix}
\circ 
\begin{bmatrix}
\fnt{F}_i^{ff} & \fnt{F}_i^{fm} \\ \fnt{F}_i^{mf} & \fnt{F}_i^{mm}
\end{bmatrix}
\circ \avg{\fnt{n}_i} }\fnt{1}.
\label{eq:onemortarmat}
\end{align}
Here, the entries of $\avg{\fnt{n}_i}$ are averages of values of the normal vector at all pairs of surface and mortar points. Both matrices, $\avg{\fnt{n}_i}$ and the block of flux matrices, are $2\times 2$ block matrices, with blocks corresponding to pairs of surface-surface points, pairs of surface-mortar points, and pairs of mortar-mortar points.
\end{remark}

\begin{figure}
\centering
\includegraphics[width=.55\textwidth]{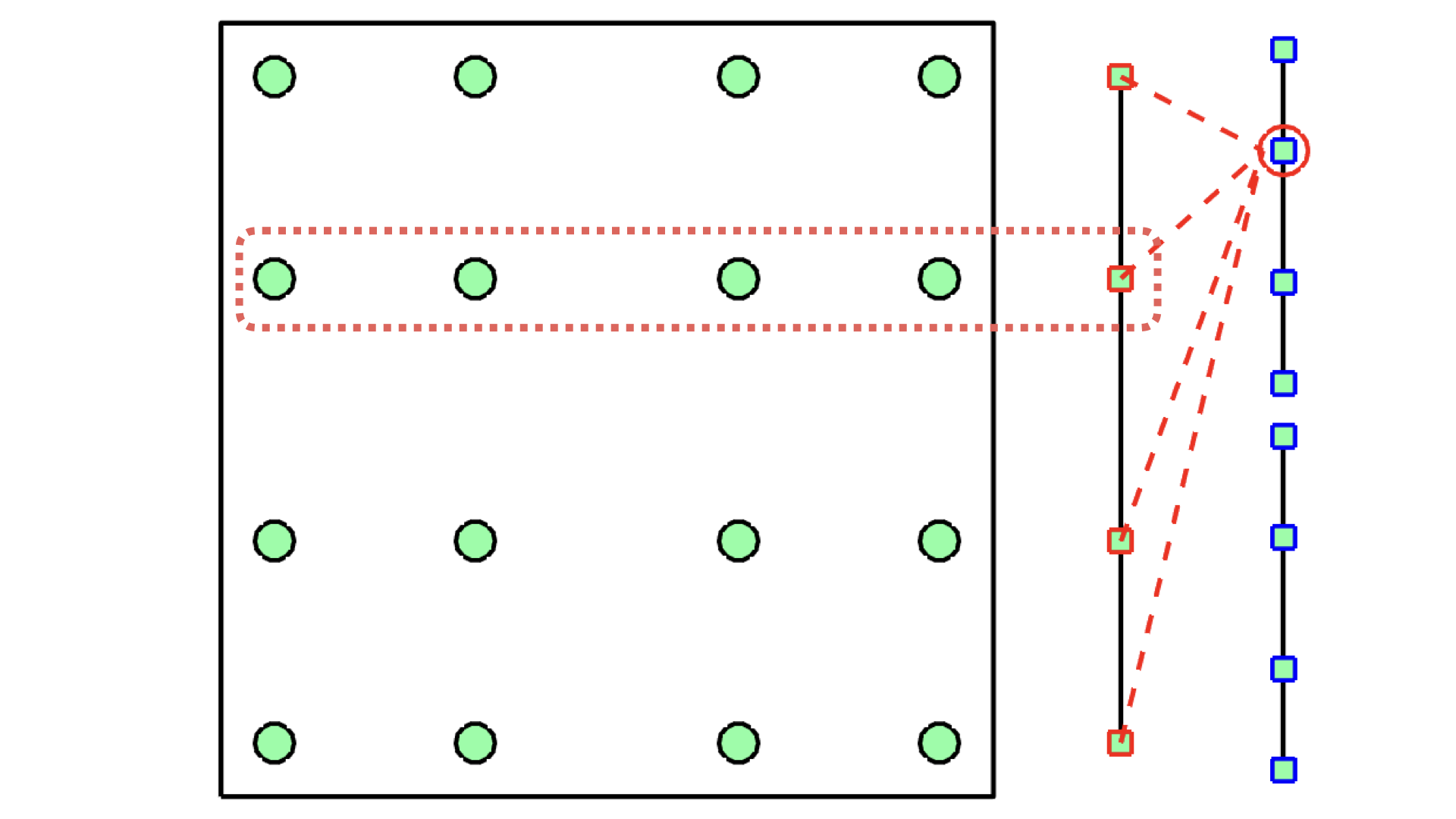}
\caption{Coupling between volume, surface, and mortar nodes for formulation (\ref{eq:esdgm_mortar}).}
\end{figure}
The formulation (\ref{eq:esdgm_mortar}) is nearly equivalent to the hybridized mortar formulation, with the main difference being the presence of skew-symmetric correction terms in the definition of the numerical flux.  Note that in this formulation, the only interactions between surface and mortar nodes occur within the computation of $\fnt{f}^*_i$, and involve only computations which are local to mortar interfaces.  
\begin{remark}
On conforming meshes, the surface and mortar nodes are the same.  Then,  
\[
\fnt{E}_{mf} = \fnt{E}_{fm} = \fnt{I}, \qquad \avg{\fnt{n}_i}_{fm} = \avg{\fnt{n}_i}_{mf}
\]
and the correction terms in the definition of $\fnt{f}^*_i$ cancel each other out.  The formulation (\ref{eq:esdgm_mortar}) then reverts to a Gauss collocation scheme \cite{chan2017discretely} on conforming meshes.  
\end{remark}

\subsection{Improving efficiency on three dimensional non-conforming hexahedral meshes}
\label{sec:anisotropic}

In 3D, the mortar-based approach enables a more efficient treatment of non-conforming hexahedral interfaces compared to existing approaches \cite{friedrich2017entropy}.  The cost of entropy-stable schemes based on flux differencing usually scales with the number of flux evaluations, which in turn scales by the number of non-zero entries present in differentiation and interpolation matrices.  For a hexahedral element, both Gauss and Lobatto node differentiation operators contain $O(N^4)$ non-zeros due to their tensor product (Kronecker) structure.  

\begin{figure}[!h]
\centering
\includegraphics[width=.5\textwidth]{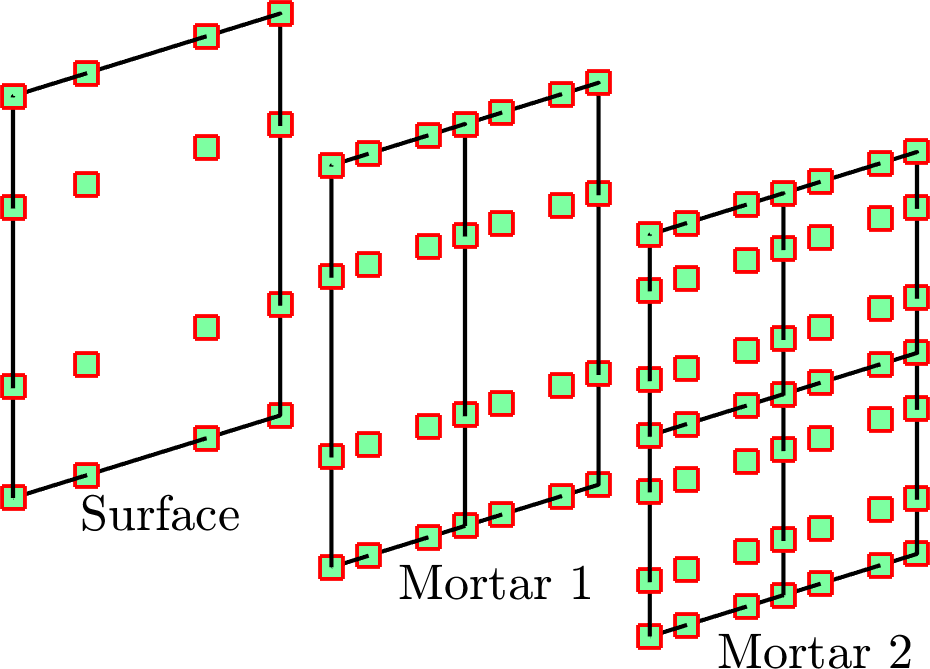}
\caption{The two-layer mortar setup for reduced $O(N^3)$ computational cost on non-conforming hexahedral interfaces.}
\label{fig:aniso}
\end{figure}

The interpolation matrices are generally dense with $O(N^4)$ non-zero entries, as they interpolate from $(N+1)^2$ surface points to $4(N+1)^2$ non-aligned mortar points (for an isotropically refined face).  However, we can expose sparsity by using a two-step interpolation from surface to mortar points.  Figure~\ref{fig:aniso} illustrates this process.  First, the $(N+1)^2$ surface points are interpolated to an anisotropically refined interface.  Because the points are aligned along one coordinate direction, this requires interpolation only along lines of nodes and results in an interpolation matrix with $O(N^3)$ non-zero entries.  The anisotropically refined nodes are then interpolated to the isotropically refined nodes.  Again, because nodes are aligned, the resulting interpolation matrix has $O(N^3)$ non-zero entries.  


This two-step procedure can be interpreted as a sparse factorization of the interpolation matrix from $(N+1)^2$ surface points to $4(N+1)^2$ points.  Let $\fnt{x}_f$ denote the surface points, and let $\fnt{x}_{m_1}, \fnt{x}_{m_2}$ denote the anisotropically and isotropically refined mortar nodes (e.g., the first and second mortar layers), respectively.  We first introduce $\fnt{E}_{\rm 1D}$ be the 1D polynomial interpolation matrix which maps from Gauss or Lobatto points to split (composite) Gauss or Lobatto points.  Let $\fnt{E}_{m_1,f}$ denote the interpolation operator from surface to nodes in the first mortar layer, and let $\fnt{E}_{m_2,m_1}$ denote the interpolation operator from nodes in the first mortar layer to the second mortar layer.  These operators are defined 
\[
\fnt{E}_{m_1,f} = \fnt{E}_{\rm 1D} \otimes \fnt{I}_{N+1}, \qquad \fnt{E}_{m_2,m_1} = \fnt{I}_{2(N+1)} \otimes \fnt{E}_{\rm 1D}.
\]
where $\fnt{I}_{N+1},\fnt{I}_{2(N+1)}$ denote the $(N+1)\times (N+1)$ and $2(N+1)\times 2(N+1)$ identity matrices, respectively.  Then, $\fnt{E}_{m_2,f} = \fnt{E}_{m_2,m_1}\fnt{E}_{m_1,f}$ interpolates from surface nodes to the second layer of mortar nodes.  However, the operators $\fnt{E}_{m_2,m_1}, \fnt{E}_{m_1,f}$ are sparse, while $\fnt{E}_{m_2,f}$ is completely dense.  

We can also define projection operators which transfer from the second to the first mortar layer and from the first mortar layer to surface nodes.  Let $\fnt{M}_{m_i} = \diag{\fnt{w}_{m_i}}$ denote the diagonal matrix containing quadrature weights $\fnt{w}_{m_i}$ for the $i$th layer of mortar nodes, respectively.  Then, we can define $\fnt{E}_{f,m_1},\fnt{E}_{m_1,m_2}$ via
\begin{align*}
\fnt{E}_{f,m_1} = \fnt{M}_f^{-1}\fnt{E}_{m_1,f}^T\fnt{M}_{m_1}\\
\fnt{E}_{m_1,m_2} = \fnt{M}_{m_1}^{-1}\fnt{E}_{m_2,m_1}^T\fnt{M}_{m_2}.
\end{align*}
The definitions of $\fnt{E}_{f,m_1}, \fnt{E}_{m_1,m_2}$ are analogous to the definition of $\fnt{E}_{fm}$ in (\ref{eq:Efm}).  The operator $\fnt{E}_{f,m_1}$ can be interpreted as projecting onto degree $N$ polynomials on a quadrilateral face, while $\fnt{E}_{m_1,m_2}$ can be interpreted as projecting onto the space of piecewise degree $N$ polynomials on the two quadrilateral faces of the first mortar layer in Figure~\ref{fig:aniso}.

Putting these pieces together allows us to define two-layer mortar-based SBP operators.  We introduce boundary matrices 
\[
\hat{\fnt{B}}_{j,m_i} = \diag{\hat{\fnt{n}}_{j,m_i}}\hat{\fnt{M}}_{m_i}, \qquad i = 1,2, \qquad j = 1,\ldots, d.
\]
where $\hat{\fnt{n}}_{j,m_i}$ denote $j$th component of the scaled outward normals evaluated at the $i$th layer of mortar points.  Then, the two-layer mortar-based SBP operator is
\begin{equation}
\hat{\fnt{Q}}_{i,m} = \frac{1}{2}\begin{bmatrix}
\hat{\fnt{Q}}_i - \hat{\fnt{Q}}_i^T & \fnt{E}^T\hat{\fnt{B}}_{i,f} & & \\
-\hat{\fnt{B}}_{i,f}\fnt{E} & & \hat{\fnt{B}}_{i,f} \fnt{E}_{f,m_1} &\\
& -\hat{\fnt{B}}_{i,m_1} \fnt{E}_{m_1,f} & & \hat{\fnt{B}}_{i,m_1}\fnt{E}_{m_1,m_2} \\
&& -\hat{\fnt{B}}_{i,m_2} \fnt{E}_{m_2,m_1} & \hat{\fnt{B}}_{i,m_2} 
\end{bmatrix}.
\end{equation}
It is straightforward to show that, for both Gauss and Lobatto nodes, $\hat{\fnt{Q}}_{i,m}$ satisfies SBP and conservation properties
\[
\hat{\fnt{Q}}_{i,m} + \hat{\fnt{Q}}_{i,m}^T = \begin{bmatrix}
\fnt{0} & & &\\
& \fnt{0} & &\\ 
& & \fnt{0} & \\
& & & \fnt{B}_{i,m_2}
\end{bmatrix}, \qquad \hat{\fnt{Q}}_{i,m}\fnt{1} = \fnt{0}.
\]
The accuracy conditions for two-layer mortar operators are also identical to those of Lemma~\ref{lemma:accuracy}.  We can now define physical SBP operators via the formula (\ref{eq:curvedQ_mortar}).  These also satisfy SBP and conservation properties under the conditions of Lemma~\ref{lemma:Qmprops_3d}.  

Using the operators introduced so far, an entropy conservative formulation using two-layer mortar operators is given by  
\begin{align}
{\fnt{M}}\td{\fnt{u}_h}{t} &+ \begin{bmatrix} \fnt{I} \\ \fnt{E} \\ \fnt{E}_{m_1,f}\fnt{E} \\ \fnt{E}_{m_2,m_1}\fnt{E}_{m_1,f}\fnt{E} \end{bmatrix}^T
\sum_{i=1}^d \LRp{2{\fnt{Q}}_{i,m} \circ \fnt{F}_i}\fnt{1} \label{eq:twomortar}\\
&+ \fnt{E}^T \fnt{E}_{m_1,f}^T\fnt{E}_{m_2, m_1}^T{\fnt{B}}_{i,m_2}\LRp{\bm{f}_{i,S}(\tilde{\fnt{u}}_{m_2}^+,\tilde{\fnt{u}}_{m_2})-\bm{f}_i(\tilde{\fnt{u}}_{m_2})} = 0,\nonumber
\end{align}
where
\[
\tilde{\fnt{u}}_{m_2} = \bm{u}(\fnt{v}_{m_2}), \qquad 
\fnt{v}_{m_2} = \fnt{E}_{m_2,m_1} \fnt{E}_{m_1,f}\fnt{v}_f.
\]
These formulations look very similar to (\ref{eq:esdgm}).  The two-layer mortar formulation can also be reformulated in a way which involves only face-local correction terms to the numerical flux:
\begin{gather}
{\fnt{M}}\td{\fnt{u}_h}{t} + \begin{bmatrix} \fnt{I} \\ \fnt{E} \end{bmatrix}^T
\sum_{i=1}^d \LRp{2{\fnt{Q}}_{i,h} \circ \fnt{F}_i}\fnt{1} + \fnt{E}^T\LRp{\fnt{f}_i^*- \fnt{B}_i\bm{f}_i(\tilde{\fnt{u}}_f)} = 0\label{eq:twomortarmat}\\
\fnt{f}_i^* = \fnt{E}_{m_1,f}^T\fnt{E}_{m_2,m_1}^T{\fnt{B}}_{i,m_2}\bm{f}_{i,S}(\tilde{\fnt{u}}_{m_2}^+,\tilde{\fnt{u}}_{m_2}) + \fnt{\delta f}_i,\nonumber\\
\fnt{\delta f}_i = \begin{bmatrix}
\fnt{I}\\
\fnt{E}_{m_1,f}\\
\fnt{E}_{m_2,m_1}\fnt{E}_{m_1,f}
\end{bmatrix}^T
\diag{\begin{bmatrix}
\fnt{w}_f\\
\fnt{w}_{m_1}\\
\fnt{w}_{m_2}
\end{bmatrix}
}
\LRp{\begin{bmatrix}
& \fnt{E}_{f,m_1} & \\
-\fnt{E}_{m_1,f} & & \fnt{E}_{m_1,m_2}\\
& -\fnt{E}_{m_2,m_1}& 
\end{bmatrix}\circ \fnt{F}_i^{f,m_1,m_2} \circ \avg{\fnt{n}_i}}\fnt{1}\nonumber
\end{gather}
where $\fnt{\delta f}_i$ is a face-local correction term for the numerical flux, and the matrix $\fnt{F}_i^{f,m_1,m_2}$ is the matrix of entropy conservative flux evaluations between pairs of solution values at all surface and mortar nodes (e.g., surface nodes and nodes in the two mortar layers).  The term $\avg{\fnt{n}_i}$ denote the averaged scaled normals between each surface point and each point in the two mortar layers.  The derivation of this face-local formulation is similar to the derivation in Section~\ref{sec:mortarimplement}, and we omit it for brevity.  

\begin{figure}[!h]
\centering
\subfloat[One-mortar flux differencing matrix ($45\times 45$ with 648 non-zero entries)]{\includegraphics[width=.45\textwidth]{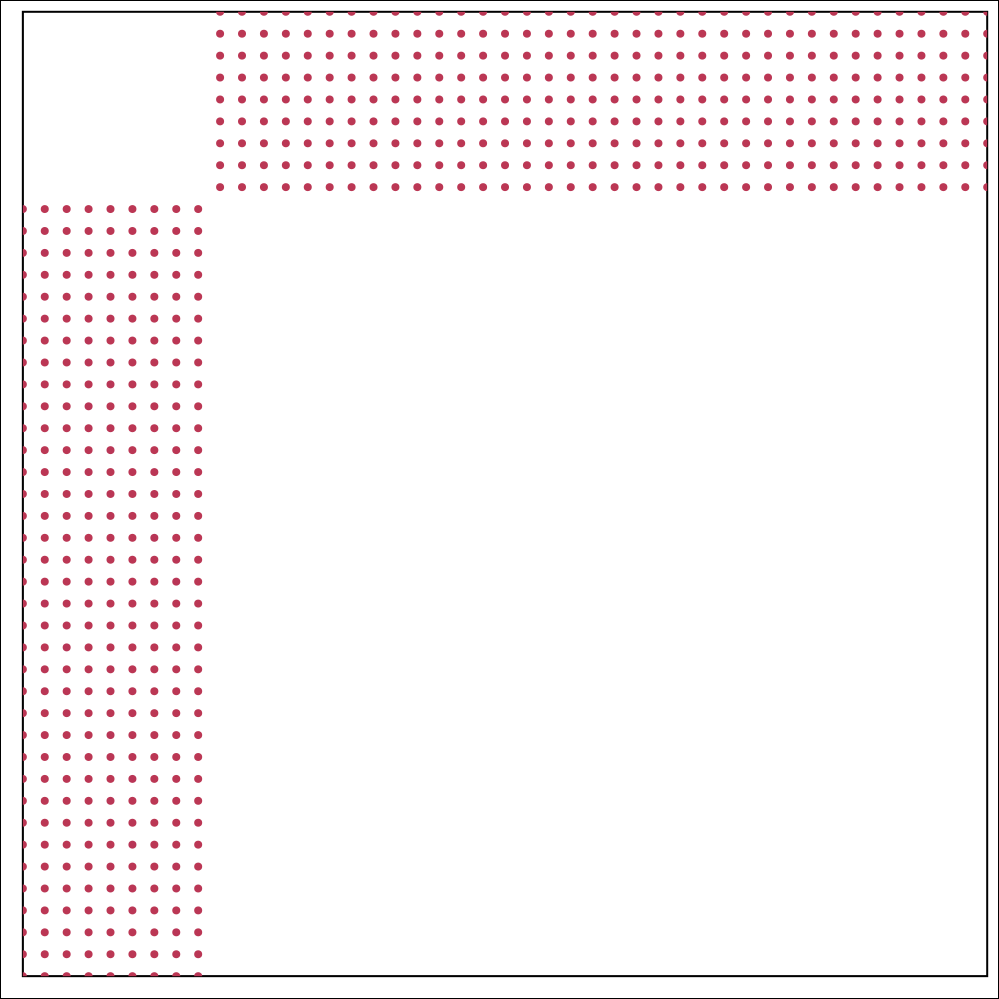}}
\hspace{1em}
\subfloat[Two-mortar flux differencing matrix ($63\times 63$ with 324 non-zero entries)]{\includegraphics[width=.45\textwidth]{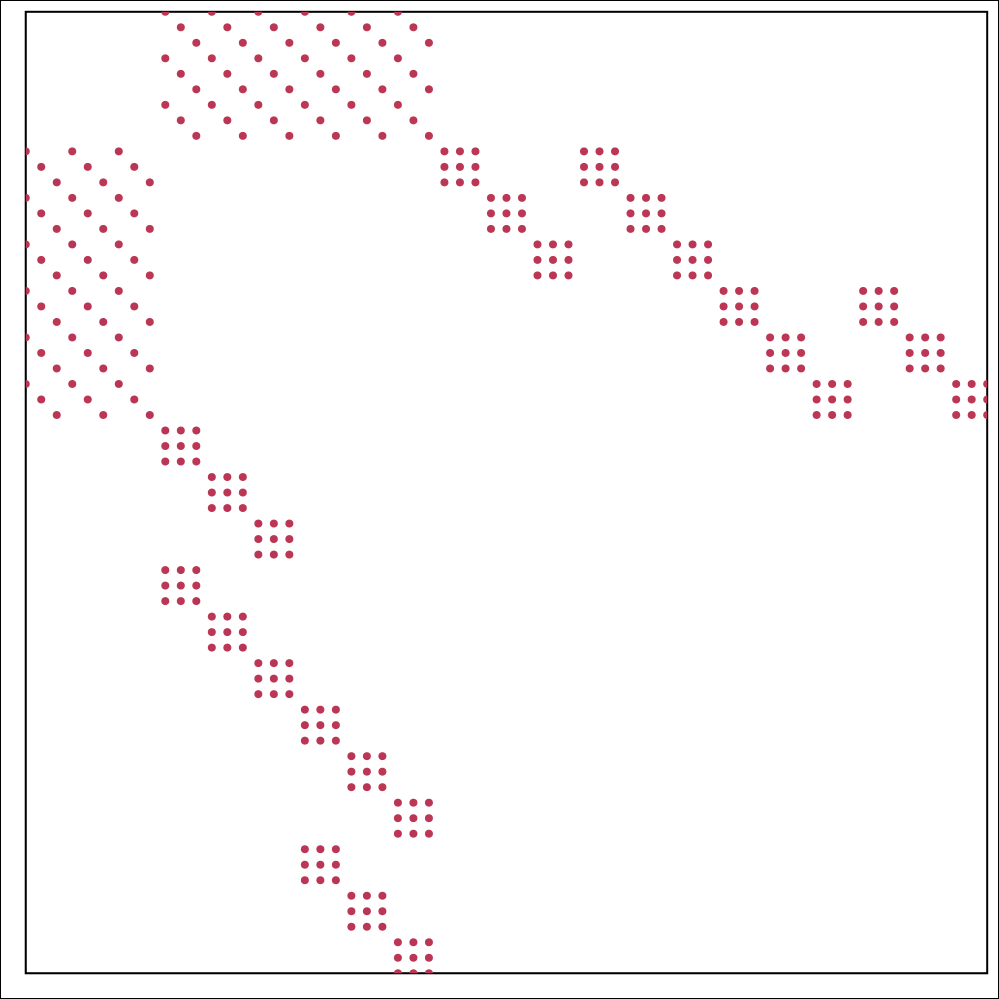}}
\caption{Sparsity pattern of flux differencing matrices used to define the mortar correction $\fnt{\delta f}_i$ for $N=2$.  The correction for the one-mortar case is given by (\ref{eq:onemortarmat}), while the correction for the two-mortar case is given by (\ref{eq:twomortarmat}).} 
\label{fig:spymortars}
\end{figure}

For both the one and two-mortar case, (\ref{eq:onemortarmat}) and (\ref{eq:twomortarmat}) define mortar corrections $\fnt{\delta f}_i$ in a way which mimics flux differencing formulated using the Hadamard product and the associated matrices
\begin{gather*}
\begin{bmatrix}
& \fnt{E}_{fm}\\
-\fnt{E}_{mf} &
\end{bmatrix} \quad \text{(one mortar)},\\
\begin{bmatrix}
& \fnt{E}_{f,m_1} & \\
-\fnt{E}_{m_1,f} & & \fnt{E}_{m_1,m_2}\\
& -\fnt{E}_{m_2,m_1}& 
\end{bmatrix} \quad \text{(two mortars)}.
\end{gather*}
The cost of flux differencing is proportional to the number of non-zero entries in each associated matrix, which favors the sparser flux differencing matrix for the two-mortar case.  Figure~\ref{fig:spymortars} shows the sparsity pattern of the one and two-mortar flux differencing matrices for $N=2$.  The size of the one-mortar flux differencing matrix is $45\times 45$ with 648 non-zero entries. While the two-mortar flux differencing matrix is slightly larger ($63\times 63$), it is significantly sparser and contains only 324 non-zero entries. As $N$ increases, the difference between the number of non-zero entries in the one and two-mortar flux differencing matrices becomes even more significant.

\section{Numerical experiments}
\label{sec:4}
This section presents numerical examples for the compressible Euler equations, which are given in $d$ dimensions as follows:
\begin{align*}
\pd{\rho}{t} &+ {\sum_{j=1}^d \pd{\LRp{\rho {u}_j}}{x_j}} = 0,\\
{\pd{\rho {u}_i}{t}} &+ {\sum_{j=1}^d \pd{\LRp{\rho u_iu_j + p\delta_{ij} }}{x_j}} = 0, \qquad i = 1,\ldots,d\\
\pd{E}{t} &+ {\sum_{j=1}^d \pd{\LRp{{u}_j(E+p)}}{x_j}} = 0.\nonumber
\end{align*}
Here, $\rho$ is density, $u_i$ denotes the $i$th component of the velocity, and $E$ is the total energy.  The pressure $p$ and specific internal energy $\rho e$ are given by 
\[
p = (\gamma-1)\LRp{E - \frac{1}{2}\rho \sum_{j=1}^d u_j^2}, \qquad 
\rho e = E - \frac{1}{2}\rho\sum_{j=1}^d u_j^2.
\]
We enforce the entropy conservation (stability) with respect to the unique entropy $S(\bm{u})$ for the viscous compressible Navier-Stokes equations \cite{hughes1986new} given by
\begin{equation*}
S(\bm{u}) = -\frac{\rho s}{\gamma-1}, 
\end{equation*}
where $s = \log\LRp{\frac{p}{\rho^\gamma}}$ is the specific entropy, and $d = 2,3$ is the dimension.  The entropy variables with respect to $S(\bm{u})$ are also derived in \cite{hughes1986new} 
\begin{align*}
v_1 = \frac{\rho e (\gamma + 1 - s) - E}{\rho e}, \qquad v_{1+ i}= \frac{\rho {{u}_i}}{\rho e}, \qquad v_{d+2} = -\frac{\rho}{\rho e}
\end{align*}
for $i = 1,\ldots, d$.  The inverse mapping is given by
\begin{align*}
\rho = -(\rho e) v_{d+2}, \qquad 
\rho {u_i} = (\rho e) v_{1+i}, \qquad 
 E = (\rho e)\LRp{1 - \frac{\sum_{j=1}^d{v_{1+j}^2}}{2 v_{d+2}}},
\end{align*}
where $i = 1,\ldots,d$, and $\rho e$ and $s$ in terms of the entropy variables are 
\begin{equation*}
\rho e = \LRp{\frac{(\gamma-1)}{\LRp{-v_{d+2}}^{\gamma}}}^{1/(\gamma-1)}e^{\frac{-s}{\gamma-1}}, \qquad 
s = \gamma - v_1 + \frac{\sum_{j=1}^d{v_{1+j}^2}}{2v_{d+2}}.
\end{equation*}

Entropy conservative fluxes involve both standard and logarithmic averages.  Let $f$ denote some piecewise continuous function, and let $f^+$ denote the exterior value of $f$ across an element face.  The average and logarithmic averages are
\[
{\avg{f} = \frac{f^+ + f}{2}}, \qquad {\avg{f}^{\log} = \frac{f^+ - f}{\log\LRp{f^+}-\log\LRp{f}}}.
\]
{The average and logarithmic average are assumed to act component-wise on vector-valued functions.  We evaluate the logarithmic average using the numerically stable algorithm of \cite{winters2019entropy}.}
Explicit expressions for entropy conservative numerical fluxes in two dimensions  are given by Chandrashekar \cite{chandrashekar2013kinetic}
\begin{align*}
&f^1_{1,S}(\bm{u}_L,\bm{u}_R) = \avg{\rho}^{\log} \avg{u_1},& &f^1_{2,S}(\bm{u}_L,\bm{u}_R) = \avg{\rho}^{\log} \avg{u_2},&\\
&f^2_{1,S}(\bm{u}_L,\bm{u}_R) = f^1_{1,S} \avg{u_1} + p_{\rm avg},&  &f^2_{2,S}(\bm{u}_L,\bm{u}_R) = f^1_{2,S} \avg{u_1},&\nonumber\\
&f^3_{1,S}(\bm{u}_L,\bm{u}_R) = f^2_{2,S},& &f^3_{2,S}(\bm{u}_L,\bm{u}_R) = f^1_{2,S} \avg{u_2} + p_{\rm avg},&\nonumber\\
&f^4_{1,S}(\bm{u}_L,\bm{u}_R) = \LRp{E_{\rm avg} + p_{\rm avg}}\avg{u_1},& &f^4_{2,S}(\bm{u}_L,\bm{u}_R) = \LRp{E_{\rm avg} + p_{\rm avg} }\avg{u_2},& \nonumber
\end{align*}
where we have introduced the auxiliary quantities $\beta = \frac{\rho}{2p}$ and  
\begin{gather}
p_{\rm avg} = \frac{\avg{\rho}}{2\avg{\beta}}, \qquad E_{\rm avg} = \frac{\avg{\rho}^{\log}}{2\avg{\beta}^{\log}\LRp{\gamma -1}}   + \frac{{u_{\rm avg}^2}}{2}, \qquad 
{u^2_{\rm avg}} = \sum_{j=1}^d u_j u_j^+ \label{eq:fluxaux}.
\end{gather}
where $u_j, u_j^+$ denote the trace and exterior values of $u_j$ across a face.  
Expressions for entropy conservative numerical fluxes for the three-dimensional compressible Euler equations can also be explicitly written as
\begin{gather*}
\bm{f}_{1,S} = \LRp{\begin{array}{c}
\avg{\rho}^{\log}\avg{u_1}\\
\avg{\rho}^{\log}\avg{u_1}^2 + p_{\rm avg}\\
\avg{\rho}^{\log}\avg{u_1}\avg{u_2}\\
\avg{\rho}^{\log}\avg{u_1}\avg{u_3}\\
(E_{\rm avg}+ p_{\rm avg})\avg{u_1}\\
\end{array}}, 
\qquad 
\bm{f}_{2,S} = \LRp{\begin{array}{c}
\avg{\rho}^{\log}\avg{u_2}\\
\avg{\rho}^{\log}\avg{u_1}\avg{u_2}\\
\avg{\rho}^{\log}\avg{u_2}^2 + p_{\rm avg}\\
\avg{\rho}^{\log}\avg{u_2}\avg{u_3}\\
(E_{\rm avg}+ p_{\rm avg})\avg{u_2}\\
\end{array}},\\
\bm{f}_{3,S} = \LRp{\begin{array}{c}
\avg{\rho}^{\log}\avg{u_3}\\
\avg{\rho}^{\log}\avg{u_1}\avg{u_3}\\
\avg{\rho}^{\log}\avg{u_2}\avg{u_3}\\
\avg{\rho}^{\log}\avg{u_3}^2 + p_{\rm avg}\\
(E_{\rm avg}+ p_{\rm avg})\avg{u_3}\\
\end{array}}. \nonumber
\end{gather*}

For all numerical experiments, we use an explicit low storage RK-45 time-stepper \cite{carpenter1994fourth} and estimate the timestep $dt$ using $J, J^k_f$, and degree-dependent $L^2$ trace constants $C_N$ 
\[
dt = C_{\rm CFL}\frac{ h}{ a C_N}, \qquad h = \frac{1}{{\nor{J^{-1}}_{L^{\infty}}}\nor{J^k_f}_{L^{\infty}}},
\]
where $a$ is an estimate of the maximum wave speed, $h$ estimates the mesh size, and $C_{\rm CFL}$ is some user-defined CFL constant.  For isotropic elements, the ratio of $J^k$ to $J^k_f$ scales as the mesh size $h$, while $C_N$ captures the dependence of the maximum stable timestep on the polynomial degree $N$.  For hexahedral elements, $C_N$ varies depending on the choice of quadrature.  It was shown in \cite{chan2015gpu} that 
\[
C_N =\begin{cases}
 d\frac{N(N+1)}{2} & \text{for GLL nodes}\\
d\frac{(N+1)(N+2)}{2} & \text{for Gauss nodes}
\end{cases}.
\]
Thus, based on this conservative estimate of the maximum stable timestep, GLL collocation schemes should be able to take a timestep which is roughly $(1 + 2/N)$ times larger than the maximum stable timestep for Gauss collocation schemes.  We do not account for this discrepancy in this work, and instead set the timestep for both GLL and Gauss collocation schemes based on the more conservative Gauss collocation estimate of $dt$.  All results are computed using $C_{\rm CFL} = 1/2$ and Lax-Friedrichs interface dissipation.  

Suppose $\fnt{u}_{\Omega}$ denotes the global solution vector; then, an entropy conservative DG scheme can be rewritten as
\[
\fnt{M}_{\Omega}\td{\fnt{u}_{\Omega}}{t} = \fnt{F}(\fnt{u}_{\Omega}).
\]
where $\fnt{M}_{\Omega}$ is the global mass matrix, and $\fnt{F}(\fnt{u}_{\Omega})$ denotes the global right hand side.  Let $\fnt{v}_{\Omega}$ denote the global vector of entropy variables.  The proof of entropy conservation implies that the ``spatial entropy'' $\fnt{v}_{\Omega}^T\fnt{F}(\fnt{u}_{\Omega}) = 0$ \cite{crean2018entropy}.   We have verified that both 2D and 3D implementations are entropy conservative (and thus entropy stable if interface dissipation is added) by computing the spatial entropy for a discontinuous initial condition on a curved mesh for $N = 1,\ldots, 4$.  In the absence of interface dissipation, all schemes yield a spatial entropy of magnitude $O\LRp{10^{-14}}$, which confirms that the implemented methods are semi-discretely entropy conservative to machine precision.

\subsection{Two-dimensional results}

We begin by examining high-order convergence of the proposed methods in two dimensions using the isentropic vortex problem \cite{shu1998essentially, crean2017high}.  The analytical solution is 
\begin{align}
\rho(\bm{x},t) &= \LRp{1 - \frac{\frac{1}{2}(\gamma-1)(\beta e^{1-r(\bm{x},t)^2})^2}{8\gamma \pi^2}}^{\frac{1}{\gamma-1}}, \qquad p = \rho^{\gamma},\\
u_1(\bm{x},t) &= 1 - \frac{\beta}{2\pi} e^{1-r(\bm{x},t)^2}(y-y_0), \qquad u_2(\bm{x},t) = \frac{\beta}{2\pi} e^{1-r(\bm{x},t)^2}(x-x_0-t),\nonumber
\end{align}
where $u_1, u_2$ are the $x$ and $y$ velocity, $r(\bm{x},t) = \sqrt{(x-x_0-t)^2 + (y-y_0)^2}$$x_0 = 5, y_0 = 0$, and $\beta = 5$.   We solve on a periodic rectangular domain $[0, 15] \times [-5,5]$ until final time $T=5$ and compute total $L^2$ errors over all solution fields.  

For a degree $N$ scheme, we approximate the $L^2$ error using an $(N+2)$ point Gauss quadrature rule.  We also examine the influence of element curvature for both GLL and Gauss collocation schemes by examining $L^2$ errors on affine and curvilinear meshes (see Figure~\ref{fig:mesh2d}).  These warpings are constructed by modifying nodal positions according to the following mapping
\begin{align*}
\tilde{x} &= x + L_x\alpha\cos\LRp{\frac{\pi}{L_x}\LRp{x-\frac{L_x}{2}}}\cos\LRp{\frac{3\pi}{L_y}y}, \\
\tilde{y} &= y + L_y\alpha\sin\LRp{\frac{4\pi}{L_x}\LRp{\tilde{x}-\frac{L_x}{2}}}\cos\LRp{\frac{\pi}{L_y}y},
\end{align*}
where $\alpha = 1/16$, and $L_x = 15, L_y = 10$ denote the lengths of the domain in the $x$ and $y$ directions, respectively.  

\begin{figure}
\centering
\includegraphics[width=.45\textwidth]{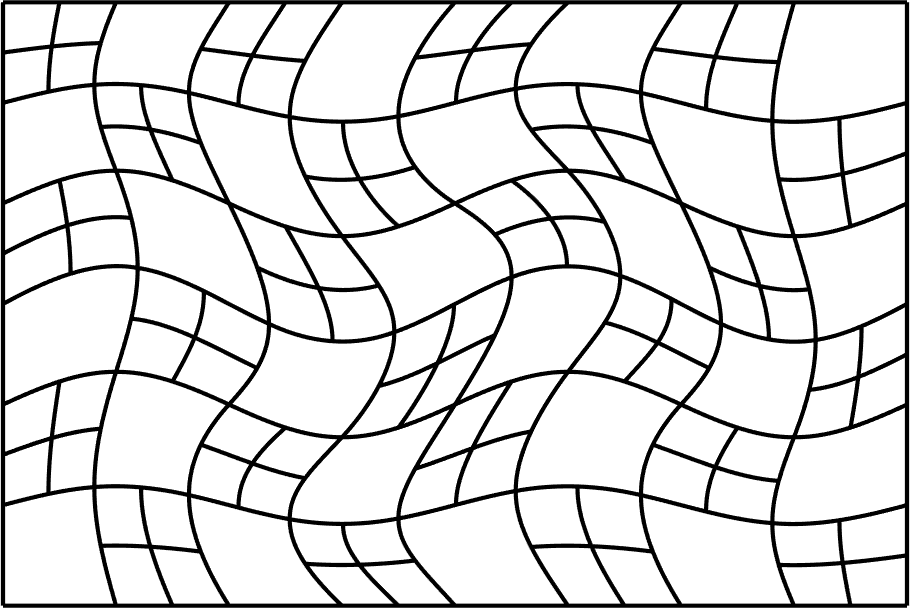}
\hspace{.5em}
\includegraphics[width=.45\textwidth]{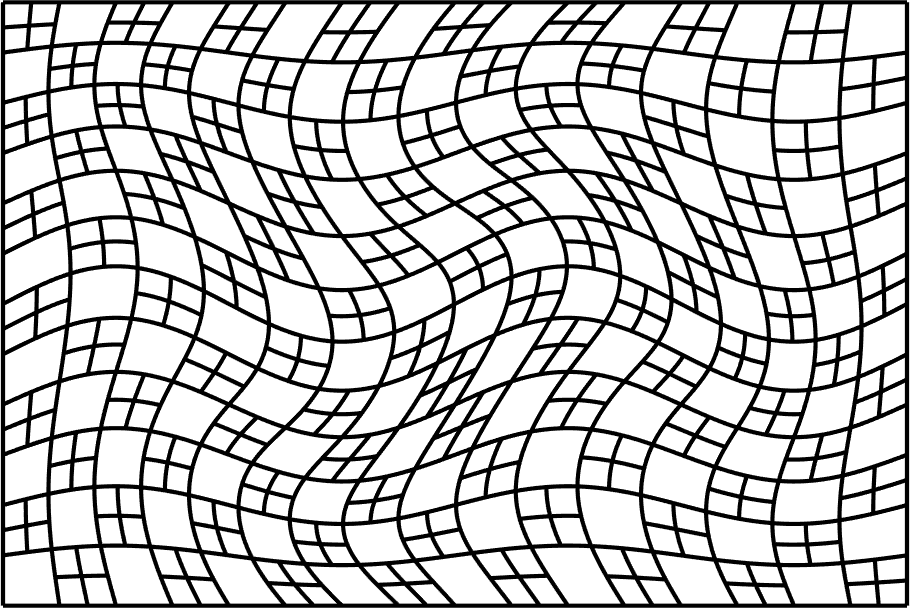}
\caption{Two curved $N=3$ meshes used in the grid refinement studies.}
\label{fig:mesh2d}
\end{figure}

Figure~\ref{fig:2d} shows results for the 2D compressible Euler equations on non-conforming quadrilateral meshes constructed by refining in a checkerboard pattern (see Figure~\ref{fig:mesh2d}), with volume, surface, and mortar quadratures constructed from one-dimensional Lobatto or Gauss quadrature rules.  As noted in Section~\ref{sec:compare} and Lemma~\ref{lemma:accuracy}, Gauss quadrature produces a degree $N$ differentiation operator, while Lobatto quadrature produces a degree $(N-1)$ differentiation operator.  We observe a similar discrepancy in $L^2$ errors for Lobatto and Gauss quadratures on non-conforming meshes in Figure~\ref{fig:2d}.

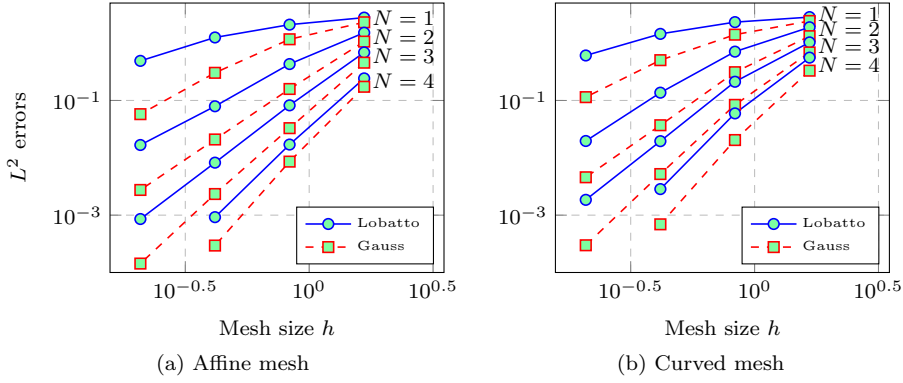
\begin{figure}
\centering
\subfloat[Affine mesh]{
\begin{tikzpicture}
\begin{loglogaxis}[
    width=.49\textwidth,
    xlabel={Mesh size $h$},
    ylabel={$L^2$ errors}, 
    xmax=3.5,
    ymin=1e-4, ymax=5,
    legend pos=south east, legend cell align=left, legend style={font=\tiny},	
    xmajorgrids=true, ymajorgrids=true, grid style=dashed,
    legend entries={Lobatto, Gauss}
]
\pgfplotsset{
cycle list={{blue, mark=*}, {red, dashed ,mark=square*}}
}
\addplot+[semithick, mark options={solid, fill=markercolor}]
coordinates{(1.66667,2.78317)(0.833333,2.09894)(0.416667,1.25478)(0.208333,0.490133)};
\addplot+[semithick, mark options={solid, fill=markercolor}]
coordinates{(1.66667,2.30772)(0.833333,1.16697)(0.416667,0.303428)(0.208333,0.0576807)};

\addplot+[semithick, mark options={solid, fill=markercolor}]
coordinates{(1.66667,1.51766)(0.833333,0.431437)(0.416667,0.0792017)(0.208333,0.0167253)};
\addplot+[semithick, mark options={solid, fill=markercolor}]
coordinates{(1.66667,1.06169)(0.833333,0.158534)(0.416667,0.0208737)(0.208333,0.00274974)};

\addplot+[semithick, mark options={solid, fill=markercolor}]
coordinates{(1.66667,0.68797)(0.833333,0.0821894)(0.416667,0.008229)(0.208333,0.000856128)};
\addplot+[semithick, mark options={solid, fill=markercolor}]
coordinates{(1.66667,0.455651)(0.833333,0.0328805)(0.416667,0.00233608)(0.208333,0.000143112)};

\addplot+[semithick, mark options={solid, fill=markercolor}]
coordinates{(1.66667,0.242584)(0.833333,0.0170238)(0.416667,0.000922618)};
\addplot+[semithick, mark options={solid, fill=markercolor}]
coordinates{(1.66667,0.173237)(0.833333,0.00858261)(0.416667,0.00029501)};

\node at (axis cs:2.4,2.8) {$N = 1$};
\node at (axis cs:2.4,1.3) {$N = 2$};
\node at (axis cs:2.4,.6) {$N = 3$};
\node at (axis cs:2.4,.22) {$N = 4$};
\end{loglogaxis}
\end{tikzpicture}
}
\hspace{.1em}
\subfloat[Curved mesh]{
\begin{tikzpicture}
\begin{loglogaxis}[
    width=.49\textwidth,
    xlabel={Mesh size $h$},
    xmax=3.5,
    ymin=1e-4, ymax=5,
    legend pos=south east, legend cell align=left, legend style={font=\tiny},	
    xmajorgrids=true, ymajorgrids=true, grid style=dashed,
    legend entries={Lobatto, Gauss}
]
\pgfplotsset{
cycle list={{blue, mark=*}, {red, dashed ,mark=square*}}
}
\addplot+[semithick, mark options={solid, fill=markercolor}]
coordinates{(1.66667,2.84018)(0.833333,2.32278)(0.416667,1.44878)(0.208333,0.608642)};
\addplot+[semithick, mark options={solid, fill=markercolor}]
coordinates{(1.66667,2.42938)(0.833333,1.40381)(0.416667,0.504184)(0.208333,0.114213)};

\addplot+[semithick, mark options={solid, fill=markercolor}]
coordinates{(1.66667,1.88997)(0.833333,0.710146)(0.416667,0.136728)(0.208333,0.0196411)};
\addplot+[semithick, mark options={solid, fill=markercolor}]
coordinates{(1.66667,1.32118)(0.833333,0.310973)(0.416667,0.0369794)(0.208333,0.00457492)};

\addplot+[semithick, mark options={solid, fill=markercolor}]
coordinates{(1.66667,1.03021)(0.833333,0.211902)(0.416667,0.0195018)(0.208333,0.00184924)};
\addplot+[semithick, mark options={solid, fill=markercolor}]
coordinates{(1.66667,0.6875)(0.833333,0.0837258)(0.416667,0.0052072)(0.208333,0.000298318)};

\addplot+[semithick, mark options={solid, fill=markercolor}]
coordinates{(1.66667,0.558299)(0.833333,0.0594828)(0.416667,0.00284697)};
\addplot+[semithick, mark options={solid, fill=markercolor}]
coordinates{(1.66667,0.330743)(0.833333,0.0204084)(0.416667,0.000691875)};

\node at (axis cs:2.4,3.25) {$N = 1$};
\node at (axis cs:2.4,1.75) {$N = 2$};
\node at (axis cs:2.4,.9) {$N = 3$};
\node at (axis cs:2.4,.425) {$N = 4$};
\end{loglogaxis}
\end{tikzpicture}
}
\caption{$L^2$ errors for Lobatto and Gauss collocation schemes on affine and curved meshes under $h$-refinement. }
\label{fig:2d}
\end{figure}

\subsection{Three-dimensional results}

We test the accuracy of the mortar-based scheme in 3D using an extruded 2D vortex propagating in the $y$ direction, with analytic expressions given in \cite{williams2013nodal}
\begin{align*}
\rho(\bm{x},t) &= \LRp{1-\frac{(\gamma-1)}{2}\Pi^2}^{\frac{1}{\gamma-1}}\\
u_i(\bm{x},t) &= \Pi r_i + \delta_{i2}, \\
E(\bm{x},t) &= \frac{p_0}{\gamma-1}\LRp{1-\frac{\gamma-1}{2}\Pi^2}^{\frac{\gamma}{\gamma-1}} + \frac{\rho}{2}\sum_{i=1}^d u_i^2.
\end{align*}
where $(u_1,u_2,u_3)^T$ is the three-dimensional velocity vector and 
\[
\Pi = \Pi_{\max}e^{\frac{1-\sum_{i=1}^dr_i^2}{2}}, \qquad \begin{pmatrix}r_1\\r_2\\r_3\end{pmatrix} = \begin{pmatrix}
-(x_2-c_2-t)\\
x_1-c_1\\
0
\end{pmatrix}.
\]
We take {$c_1 = c_2 = 7.5$}, $p_0 = {1}/{\gamma}$, and $\Pi_{\max} = 0.4$, and solve on the domain $[0,15]\times [0,20]\times [0,1]$.  We construct a non-conforming mesh by gluing together two uniform hexahedral meshes of different resolutions (see Figure~\ref{fig:noncon3dmesh}).  The non-conforming interface runs parallel to the $y$-axis, such that the vortex propagates along the direction of the non-conforming interface.  The solution is evolved until final time $T=1$.

\begin{figure}
\centering
\includegraphics[width=.65\textwidth]{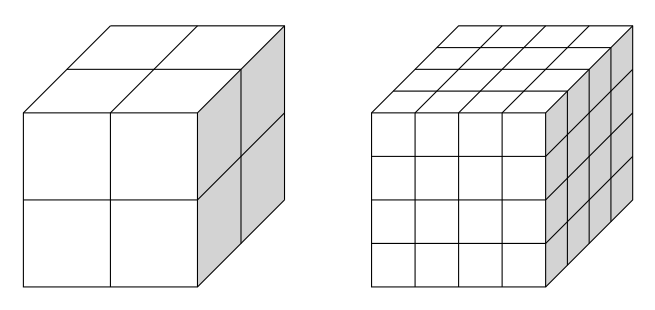}
\caption{Construction of a non-conforming hexahedral mesh by combining two conforming Cartesian meshes. }
\label{fig:noncon3dmesh}
\end{figure}

First, we confirmed that the use of two mortar layers does not impact the accuracy of the scheme by computing computed errors using both one and two mortar layers.  For $N=1,2,3$, the $L^2$ and $L^\infty$ errors for the one and two-mortar cases agree up to the $8$th significant digit on a sequence of four uniformly refined meshes, and we conclude that any differences in accuracy between the one and two-mortar cases are negligible.  

Figure~\ref{fig:3d_affine} shows computed estimates of $L^2$ and $L^{\infty}$ errors, where we estimate both by evaluating the solution using a tensor product Gauss quadrature rule with $N+3$ points in each coordinate direction.  For the $L^2$ error, we observe that Gauss quadrature is more accurate than Lobatto quadrature.  However, Gauss quadrature delivers significantly more accurate $L^\infty$ errors than Lobatto quadrature, which is consistent with numerical experiments in \cite{fernandez2019entropy}.  

Finally, we discuss computed $L^2$ convergence rates.  For $N=1$, we observe convergence rates of $1.365$ and $1.939$ for Lobatto and Gauss, respectively.  This order reduction is less prominent for $N=3$, with Lobatto and Gauss quadratures achieving rates of $3.468$ and $3.977$, respectively.  Order reduction is observed for both Gauss and Lobatto nodes for $N=2$, with both choices achieving convergence rates of around $2.3$.  This may indicate that the meshes have not yet been sufficiently refined to reach the asymptotic regime, or may be a consequence of the use of Lax-Friedrichs penalization \cite{hindenlang2019order}.

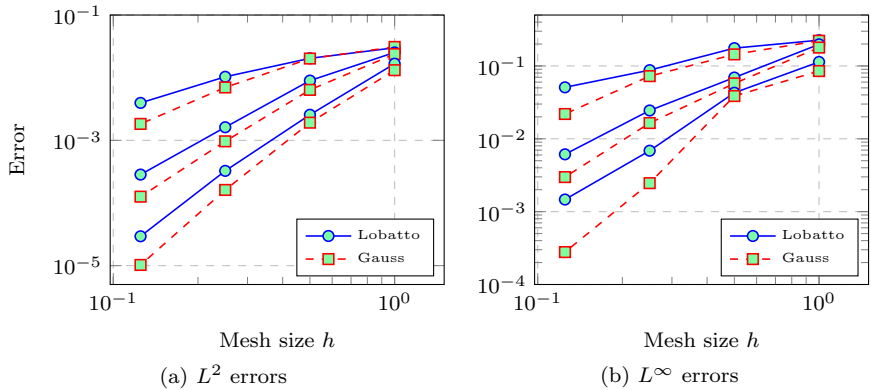
\begin{figure}
\centering
\subfloat[$L^2$ errors]{
\begin{tikzpicture}
\begin{loglogaxis}[
    width=.49\textwidth,
    xlabel={Mesh size $h$},
    ylabel={Error}, 
    xmax=1.5,
    ymin=5e-6, ymax=.1,
    legend pos=south east, legend cell align=left, legend style={font=\tiny},	
    xmajorgrids=true, ymajorgrids=true, grid style=dashed,
    legend entries={Lobatto, Gauss}
]
\pgfplotsset{
cycle list={{blue, mark=*}, {red, dashed ,mark=square*}}
}
\addplot+[semithick, mark options={solid, fill=markercolor}]
coordinates{(1,0.0300089)(0.5,0.02037)(0.25,0.0103349)(0.125,0.00395744)};
\addplot+[semithick, mark options={solid, fill=markercolor}]
coordinates{(1,0.0309619)(0.5,0.0201996)(0.25,0.00698411)(0.125,0.0018293)};

\addplot+[semithick, mark options={solid, fill=markercolor}]
coordinates{(1,0.0252826)(0.5,0.00898116)(0.25,0.00161751)(0.125,0.000284006)};
\addplot+[semithick, mark options={solid, fill=markercolor}]
coordinates{(1,0.0237929)(0.5,0.00638897)(0.25,0.000964247)(0.125,0.000126189)};

\addplot+[semithick, mark options={solid, fill=markercolor}]
coordinates{(1,0.0167185)(0.5,0.00255939)(0.25,0.000325533)(0.125,2.94146e-05)};
\addplot+[semithick, mark options={solid, fill=markercolor}]
coordinates{(1,0.0131715)(0.5,0.0019141)(0.25,0.000161528)(0.125,1.02566e-05)};

\node at (axis cs:2.4,2.8) {$N = 1$};
\node at (axis cs:2.4,1.3) {$N = 2$};
\node at (axis cs:2.4,.6) {$N = 3$};
\end{loglogaxis}
\end{tikzpicture}
}
\hspace{.1em}
\subfloat[$L^{\infty}$ errors]{
\begin{tikzpicture}
\begin{loglogaxis}[
    width=.49\textwidth,
    xlabel={Mesh size $h$},
    xmax=1.5,
    ymin=1e-4, ymax=.5,
    legend pos=south east, legend cell align=left, legend style={font=\tiny},	
    xmajorgrids=true, ymajorgrids=true, grid style=dashed,
    legend entries={Lobatto, Gauss}
]
\pgfplotsset{
cycle list={{blue, mark=*}, {red, dashed ,mark=square*}}
}
\addplot+[semithick, mark options={solid, fill=markercolor}]
coordinates{(1,0.225732)(0.5,0.175585)(0.25,0.0875643)(0.125,0.050868)};
\addplot+[semithick, mark options={solid, fill=markercolor}]
coordinates{(1,0.221903)(0.5,0.144126)(0.25,0.0724546)(0.125,0.0219034)};

\addplot+[semithick, mark options={solid, fill=markercolor}]
coordinates{(1,0.198065)(0.5,0.0693157)(0.25,0.024472)(0.125,0.00612038)};
\addplot+[semithick, mark options={solid, fill=markercolor}]
coordinates{(1,0.178736)(0.5,0.0575655)(0.25,0.016464)(0.125,0.0029831)};

\addplot+[semithick, mark options={solid, fill=markercolor}]
coordinates{(1,0.113556)(0.5,0.0428492)(0.25,0.00686179)(0.125,0.00146676)};
\addplot+[semithick, mark options={solid, fill=markercolor}]
coordinates{(1,0.0851151)(0.5,0.0386236)(0.25,0.00246054)(0.125,0.00027861)};

\node at (axis cs:2.4,3.25) {$N = 1$};
\node at (axis cs:2.4,1.75) {$N = 2$};
\node at (axis cs:2.4,.9) {$N = 3$};
\end{loglogaxis}
\end{tikzpicture}
}
\caption{$L^2$ and $L^\infty$ errors for entropy-stable Lobatto and Gauss collocation schemes on non-conforming Cartesian (affine) hexahedral meshes under $h$-refinement for $N=1,2,3$.  Note the different axis limits.}
\label{fig:3d_affine}
\end{figure}


\section{Conclusions and extensions of the mortar-based approach}

We have introduced an entropy-stable mortar-based interface treatment which reduces the computational cost of entropy-stable high-order Lobatto and Gauss collocation DG schemes on non-conforming meshes.  Mortars are incorporated into entropy-stable DG schemes using mortar-based extensions of hybridized SBP operators introduced in \cite{chan2018discretely, chan2018efficient, chan2019skew}, and admit efficient face-local implementations. Moreover, in addition to extending Gauss collocation DG schemes to non-conforming mesh interfaces, the mortar-based entropy-stable formulations introduced in this work also provide a more efficient treatment of non-conforming interfaces for Lobatto collocation DG schemes based on a two-mortar setup. 

We conclude by discussing some extensions and generalizations of this mortar-based approach.  First, the extension to $p$ and $hp$ non-conforming meshes is straightforward.  As with $h$ non-conforming meshes, the mortar ndoes are taken to be the nodes on the refined side of a $hp$ non-conforming interfaces.  The two-layer mortar approach of Section~\ref{sec:anisotropic} can also be extended to $hp$ non-conforming interfaces in a straightforward fashion.  

This approach can also be generalized to non-hexahedral elements.  However, the mortar-based SBP operators in this paper require node alignment to expose more sparsity.  Unlike tensor product elements, simplicial elements, neither mortar nor surface nodes are ``aligned'' with volume nodes.  Thus, the mortar-based approach in this work is not expected to be more efficient than the direct use of hybridized SBP operators on non-tensor product elements \cite{shadpey2019energy}.

\section{Acknowledgments}

Jesse Chan gratefully acknowledges support from the National Science Foundation under awards DMS-1719818 and DMS-1712639.  Mario Bencomo is supported partially by the Rice University CAAM Department Pfieffer Postdoctoral Fellowship.  The authors also thank Florian Hindenlang for his careful reading of the initial manuscript.

\appendix

\section{Geometric terms for curved non-conforming meshes}
\label{app:A}

Entropy conservation (stability) requires that the geometric terms be constructed appropriately.  In addition to requiring satisfaction of the discrete geometric conservation law (GCL) \cite{kopriva2006metric, crean2018entropy, chan2018discretely}, the schemes derived in this work require that the polynomial degree of the geometric terms be related to the accuracy of the surface and mortar quadrature rules.  

As discussed in Section~\ref{sec:compare}, the standard construction of geometric terms (using a cross product formula \cite{kopriva2006metric, hesthaven2007nodal}) on 2D isoparametric curved elements are sufficient to ensure the GCL and thus entropy conservation (stability) for Lobatto and Gauss collocation schemes.  For 3D isoparametric curved hexahedra, the standard construction of geometric terms does not ensure satisfaction of the GCL.  We consider two alternative methods to construct geometric terms: Approach 1 from \cite{kopriva2006metric}, and Approach 2 from \cite{kozdon2018energy}.  

We note that entropy conservation (or stability) on non-conforming meshes also requires that the scaled outward normals are equal and opposite across each mortar face.  This holds if the mesh is \emph{watertight} or \emph{well-constructed}, such that there are no gaps between neighboring elements.  Several conditions which guarantee that a curved non-conforming mesh is well-constructed are described in \cite{kopriva2019free}.  For the two-dimensional numerical experiments in this work, we construct watertight curved non-conforming meshes by first curving a conforming mesh.  We locally refine to construct a non-conforming mesh, with the refined child elements inheriting the geometric mapping of the parent elements.  For the three-dimensional numerical experiments in this work, we adopt the following approach.

\subsection{Construction of geometric terms for curved hexahedra}

Let $\fnt{D}_{\rm 1D}$ denote the $(N_{\rm geo}+1) \times (N_{\rm geo}+1)$ one-dimensional nodal differentiation matrix at degree $N_{\rm geo}$ Lobatto nodes.  Multi-dimensional differentiation matrices on the reference hexahedron can then be constructed via
\begin{align*}
\fnt{D}_1 &= \fnt{D}_{\rm 1D} \otimes \fnt{I}_{\rm 1D}\otimes \fnt{I}_{\rm 1D}\\
\fnt{D}_2 &= \fnt{I}_{\rm 1D} \otimes \fnt{D}_{\rm 1D}\otimes \fnt{I}_{\rm 1D}\\
\fnt{D}_3 &= \fnt{I}_{\rm 1D} \otimes \fnt{I}_{\rm 1D}\otimes \fnt{D}_{\rm 1D},
\end{align*}
where $\fnt{I}_{\rm 1D}$ is the $(N_{\rm geo}+1) \times (N_{\rm geo}+1)$ identity matrix.  

Let $\fnt{x},\fnt{y},\fnt{z}$ denote the vectors of physical nodal coordinates of a curved degree $N_{\rm geo}$ element.  Approach 1 in \cite{kopriva2006metric} constructs scaled geometric terms $g_{ij} = J\pd{\hat{x}_j}{x_i}$ by computing the curl of geometric ``potentials'' $\fnt{f}_{ij}$
\begin{align*}
\fnt{g}_{i1} &= \alpha_i\LRp{\fnt{D}_3\fnt{f}_{i2} - \fnt{D}_2\fnt{f}_{i3}},\\
\fnt{g}_{i2} &=\alpha_i \LRp{\fnt{D}_1\fnt{f}_{i3} - \fnt{D}_3\fnt{f}_{i1}},\\
\fnt{g}_{i3} &= \alpha_i\LRp{\fnt{D}_2\fnt{f}_{i1} - \fnt{D}_1\fnt{f}_{i2}}
\end{align*}
where $\alpha_1 = 1$ and $\alpha_2 = \alpha_3 = -1$, and the potentials $\fnt{f}_{ij}$ are defined as 
\begin{align}
\fnt{f}_{1j} = \LRp{ \fnt{D}_j \fnt{y}} \circ \fnt{z}, \qquad 
\fnt{f}_{2j} = \LRp{ \fnt{D}_j \fnt{x}} \circ \fnt{z}, \qquad
\fnt{f}_{3j} = \LRp{ \fnt{D}_j \fnt{y}} \circ \fnt{x}.  \label{eq:geopot}
\end{align}
The resulting geometric terms are tensor product polynomials of degree $g_{ij} \in Q^{N_{\rm geo}}$.  We note that any geometric terms constructed by applying the discrete curl operator automatically satisfy the geometric conservation law.  

On non-conforming meshes, the resulting construction does not guarantee that the scaled normal vectors (\ref{eq:nJ_Gnhat}) are equal and opposite across faces \cite{kopriva2019free}.  This is remedied by enforcing continuity of the geometric potentials $\fnt{f}_{ij}$ across all non-conforming interfaces (for example, by interpolating values of $\fnt{f}_{ij}$ from the coarse side to the refined side).  

Approach 2 in \cite{kozdon2018energy} (see also \cite{chan2019skew}, Footnote 3) further modifies this approach by slightly by reducing the degree of each geometric potential along the corresponding coordinate direction.  Let $\fnt{I}_{N_{\rm geo}}^{N_{\rm geo}-1}$ denote the interpolation operator from degree $N_{\rm geo}$ to degree $N_{\rm geo}-1$ Lobatto nodes, and let $\fnt{I}_{N_{\rm geo}-1}^{N_{\rm geo}}$ denote the interpolation operator from degree $N_{\rm geo}-1$ to degree $N_{\rm geo}$ Lobatto nodes.  Define the degree reduction operator $\fnt{F}_{\rm 1D}$ as the product of these interpolation operators
\[
\fnt{F}_{\rm 1D} = \fnt{I}_{N_{\rm geo}-1}^{N_{\rm geo}}\fnt{I}_{N_{\rm geo}}^{N_{\rm geo}-1}.
\]
Multiplication by $\fnt{F}_{\rm 1D}$ lowers the degree of a polynomial from $N_{\rm geo}$ to $N_{\rm geo}-1$ while leaving the boundary values unchanged.  We can define multi-dimensional degree reduction operators 
\begin{align*}
\fnt{F}_1 &= \fnt{F}_{\rm 1D} \otimes \fnt{I}_{\rm 1D}\otimes \fnt{I}_{\rm 1D}\\
\fnt{F}_2 &= \fnt{I}_{\rm 1D} \otimes \fnt{F}_{\rm 1D}\otimes \fnt{I}_{\rm 1D}\\
\fnt{F}_3 &= \fnt{I}_{\rm 1D} \otimes \fnt{I}_{\rm 1D}\otimes \fnt{F}_{\rm 1D}.
\end{align*}
Suppose $u \in Q^{N_{\rm geo}}$ is represented using nodal values $\fnt{u}$; then, $\fnt{F}_1\fnt{u}$ corresponds to a polynomial in $Q^{N_{\rm geo}-1,N_{\rm geo},N_{\rm geo}}$.  In other words, $\fnt{F}_i$ lowers the polynomial degree in the $i$th coordinate by $1$.  Approach 2 computes the geometric terms $\fnt{g}_{ij}$ via
\begin{align*}
\fnt{g}_{i1} &= \alpha_i\LRp{\fnt{D}_3\tilde{\fnt{f}}_{i2} - \fnt{D}_2\tilde{\fnt{f}}_{i3}},\\
\fnt{g}_{i2} &=\alpha_i \LRp{\fnt{D}_1\tilde{\fnt{f}}_{i3} - \fnt{D}_3\tilde{\fnt{f}}_{i1}},\\
\fnt{g}_{i3} &= \alpha_i\LRp{\fnt{D}_2\tilde{\fnt{f}}_{i1} - \fnt{D}_1\tilde{\fnt{f}}_{i2}}
\end{align*}
where $\tilde{\fnt{f}}_{ij} = \fnt{F}_j \fnt{f}_{ij}$ are degree-reduced geometric potentials, and the original geometric potentials $\fnt{f}_{ij}$ are computed via (\ref{eq:geopot}).  Since multiplication by $\fnt{D}_i$ reduces the degree by $1$ in the $i$th coordinate by $1$, the resulting geometric terms $g_{ij}$ satisfy
\begin{align*}
&g_{i1} \in Q^{N_{\rm geo},N_{\rm geo}-1,N_{\rm geo}-1}\\
&g_{i2} \in Q^{N_{\rm geo}-1,N_{\rm geo},N_{\rm geo}-1}\\
&g_{i3} \in Q^{N_{\rm geo}-1,N_{\rm geo}-1,N_{\rm geo}}, \qquad i = 1,2,3.
\end{align*} 
The difference between Approach 1 and Approach 2 appears in Lemma~\ref{lemma:Qmprops_3d}.  For Lobatto nodes, the mortar-based operators $\fnt{Q}_{i,m}$ satisfy an SBP property (and can be used to construct entropy-stable formulations) for $N_{\rm geo} < N$ under Approach 1, while Approach 2 relaxes this condition to $N_{\rm geo} \leq N$.  

While Approach 1 is marginally more accurate than Approach 2 for $N > 1$, both approaches achieve nearly identical accuracy and convergence rates.  We analyze $L^2$ errors using Approach 1 and Approach 2 by  computing exact values of $g_{ij}$ using the cross product formula \cite{kopriva2006metric, hesthaven2007nodal}.   We construct a curved mesh by transforming a Cartesian hexahedral grid on $[-1,1]^3$ by interpolating the global curved mapping at Lobatto nodes
\begin{align*}
    \tilde{x} &= x + \frac{1}{4}\cos(x)\sin(y)\sin(z)\\
    \tilde{y} &= y + \frac{1}{4}\sin(x)\cos(y)\sin(z)\\
    \tilde{z} &= z + \frac{1}{4}\sin(x)\sin(y)\cos(z),
\end{align*}
where $x,y,z$ denote coordinates on the Cartesian grid and $\tilde{x},\tilde{y},\tilde{z}$ denote coordinates on the mapped curved domain.  Figure~\ref{fig:geo_err} shows the convergence of the $L^2$ error over all geometric terms under mesh refinement.  Both Approach 1 and Approach 2 achieve very similar errors for $N=1,\ldots, 4$, with Approach 1 achieving slightly lower errors for all $N > 1$.  Table~\ref{fig:geo_rates} shows computed convergence rates, which match the $O(h^{N+2})$ rates proven in \cite{chan2018discretely, crean2018entropy} for simplicial elements.
\begin{figure}
\centering
\begin{tikzpicture}
\begin{loglogaxis}[
    width=.5\textwidth,
    xlabel={Mesh size $h$},
    ylabel={$L^2$ error}, 
    xmax=1.7,xmin=.055,
    ymin=2.5e-8, ymax=2.5,
    legend pos=south east, legend cell align=left, legend style={font=\tiny},	
    xmajorgrids=true, ymajorgrids=true, grid style=dashed,
    legend entries={Approach 1 \cite{kopriva2006metric}, Approach 2 \cite{kozdon2018energy}}
]
\pgfplotsset{
cycle list={{blue, mark=*}, {red, dashed ,mark=square*}}
}
\addplot+[semithick, mark options={solid, fill=markercolor}]
coordinates{(1,0.576948)(0.5,0.147753)(0.25,0.0217505)(0.125,0.00284466)};
\addplot+[semithick, mark options={solid, fill=markercolor}]
coordinates{(1,0.812383)(0.5,0.207677)(0.25,0.0306257)(0.125,0.00401036)};

\addplot+[semithick, mark options={solid, fill=markercolor}]
coordinates{(1,1.14476)(0.5,0.0680173)(0.25,0.00414906)(0.125,0.000257077)};
\addplot+[semithick, mark options={solid, fill=markercolor}]
coordinates{(1,0.885454)(0.5,0.0505198)(0.25,0.0029414)(0.125,0.000177322)};

\addplot+[semithick, mark options={solid, fill=markercolor}]
coordinates{(1,0.178578)(0.5,0.00711309)(0.25,0.000232189)(0.125,7.33457e-06)};
\addplot+[semithick, mark options={solid, fill=markercolor}]
coordinates{(1,0.163661)(0.5,0.00592117)(0.25,0.000187325)(0.125,5.866e-06)};

\addplot+[semithick, mark options={solid, fill=markercolor}]
coordinates{(1,0.0419636)(0.5,0.000598623)(0.25,8.70891e-06)(0.125,1.31374e-07)};
\addplot+[semithick, mark options={solid, fill=markercolor}]
coordinates{(1,0.0380338)(0.5,0.000530832)(0.25,7.64254e-06)(0.125,1.14732e-07)};

\node at (axis cs:.085,3.5e-3) {$N = 1$};
\node at (axis cs:.085,2.2e-4) {$N = 2$};
\node at (axis cs:.085,6e-6) {$N = 3$};
\node at (axis cs:.085,1.25e-7) {$N = 4$};
\end{loglogaxis}
\end{tikzpicture}
\caption{$L^2$ errors under mesh refinement for Approach 1 \cite{kopriva2006metric} and Approach 2 \cite{kozdon2018energy}.}
\label{fig:geo_err}
\end{figure}
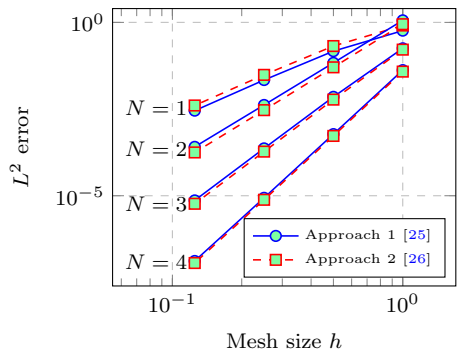

\begin{table}
\centering
\begin{tabular}{|c||c|c|c|c|}
\hline
$N$ & 1 & 2 & 3 & 4\\
\hline
Approach 1 &2.8472  & 4.0772 &4.9896 & 6.0879 \\
\hline
Approach 2 & 2.8494  & 4.0238 & 4.9608 & 6.0769  \\
\hline
\end{tabular}
\caption{Computed $L^2$ error convergence rates for Approach 1 \cite{kopriva2006metric} and Approach 2 \cite{kozdon2018energy}.  Both approaches achieve $O(h^{N+2})$ rates, which were proven in \cite{chan2018discretely, crean2018entropy} for simplicial elements.}
\label{fig:geo_rates}
\end{table}

\bibliographystyle{unsrt}
\bibliography{refs}

\end{document}